\documentclass{article}
\usepackage{hyperref}
\usepackage{amsmath, amsfonts}
\usepackage[dvips]{graphicx}
\usepackage{color} 
\usepackage{amsmath, amsfonts}

\usepackage[left=2cm, top=.75cm, bottom=0.5cm,right=2cm]{geometry}

\newcommand{\bl}{\textcolor[rgb]{0.00,0.00,0.00}}

\newcommand{\rd}{\textcolor[rgb]{0,0,0}}

\usepackage{amsmath, amsfonts}

\usepackage{amsthm}

\theoremstyle{plain}
\newtheorem{thm}{Theorem}[section]

\newtheorem{lem}[thm]{Lemma} 
\newtheorem{prop}[thm]{Proposition}

\theoremstyle{definition}
\newtheorem{defn}{Definition}[section]
\newtheorem{Assumption}[thm]{Assumption}

\theoremstyle{remark}
\newtheorem{rem}{Remark}[section]

\newcommand{\upperRomannumeral}[1]{\uppercase\expandafter{\romannumeral#1}}
\newcommand{\be}{\begin{equation}}

\newcommand{\ee}{\end{equation}}
\newcommand{\bq}{\begin{eqnarray}}
\newcommand{\eq}{\end{eqnarray}}

\newcommand{\half}{\frac{1}{2}}

\newcommand{\nn}{\nonumber}

\newcommand{\gm}{\gamma}
\newcommand{\lb}{\lbrace}
\newcommand{\rb}{\rbrace}

\newcommand{\sk}{\smallskip}
\newcommand{\la}{\langle}
\newcommand{\ra}{\rangle}

\newcommand{\e}{\varepsilon}
\newcommand{~}{\,\,\,\,=\,\,\,\,}
\newcommand{\lee}{\,\,\,\,\, \le \,\,\,\,\,}

\newcommand{\lm}{\lambda}
\newcommand{\Lm}{\Lambda }

\newcommand{\al}{\alpha}

\newcommand{\mc}{\mathcal}
\newcommand{\Ex}{\mathbb{E}}

\begin{document}
	\author{
		Alessandro Bondi\thanks{Department of AI, Data and Decision Sciences, Luiss University, Rome. {\tt abondi@luiss.it
		}}
		\and
		Martin Forde\thanks{Department of Mathematics, King's College London. {\tt https://martinforde.github.io}}
		%\small 31st May 2025
	}
	
	%\maketitle
	%\sk 
	\title{\textbf{
			%Weak convergence for a fast mean-reverting rough Heston model 
			%with jumps to a general class of L\'{e}vy processes
			\rd{L\'{e}vy processes as weak limits  of rough Heston models}}
		% - the case $H\in (0,-\half)$
	}
	\maketitle
		
		\begin{abstract}
			We show weak convergence of the time-$t$ marginals for the integrated variance in a re-scaled rough Heston model to an Inverse Gaussian L\'{e}vy process.    
			This shows we can obtain such a limit without having to impose that the true Hurst exponent $H$ for the model is $\half$ as in \cite{AC24}, %\footnote{The $H$ parameter in \cite{AC24} is not fixed to be $\half$, but the Hurst exponent for their family of $V^{\e}$ processes is of course $\half$ since the models are all standard Heston for $\e>0$.} 
			or that $H\searrow -\half$ as in \cite{AAR25}, so the result potentially has increased financial relevance.
			We later extend the analysis to the case where $V$ has jumps, showing weak convergence of the finite-dimensional distributions of the integrated variance to a deterministic time-change of the first-passage time process to lower barriers for a more general class of spectrally positive L\'evy processes. This convergence result is then strengthened to a functional setting, namely on the space of c\`adl\`ag functions on the non-negative half-line endowed with the $M_1$ topology.
			\vspace{2mm}
			\vspace{1mm}
		\end{abstract}

        \section{Introduction}\label{sec_Intro}
 \quad  
 Stochastic Volterra equations (SVEs) arise in a variety of areas of applied mathematics, as they provide a natural framework for modeling systems with memory and irregular behavior through the presence of  kernels that drive the dynamics. For instance, they appear as scaling limits of branching processes in chemical and biological interaction models \cite{MS}, or of Hawkes processes in mathematical finance \cite{EFR18}. Volterra-type L\'evy processes are also used in the stochastic modeling of energy markets; see \cite{BBV}. 
 In general, SVEs generate processes that are neither Markov processes nor semimartingales; nevertheless, they provide a theoretically convenient framework that is currently being explored in several directions, some of which are not primarily motivated by applications. For instance, general results on the theory of SVEs, in both the continuous and jump settings, can be found in \cite{A21,ACLP,ALP19,AS, BLP24}. The invariance theory for SVEs has also attracted considerable attention in recent years; see \cite{Alfonsi,ACPPS,BP_feller}. We also refer to \cite{BF,CT20,HA, Z} for studies of SVEs in infinite dimension.  
\sk
 
 In this work, our analysis starts from a Volterra square-root diffusion $V=(V_t)_{t\ge0}$, which is a non-negative process satisfying the dynamics
 \begin{equation}\label{SVE_cont}
    \begin{aligned}
       &V_t =  V_0 \,+\,
	\int_0^t K(t-s) \Big(\lm(\theta-V^{}_s)ds+\sigma\sqrt{V_s} dW_s\Big), 
    \\&\qquad \text{where \quad $K(t)=\frac{1}{\Gamma(\alpha)}t^{\al-1},$ for $\alpha\in \Big(\frac{1}{2},1\Big)$.}
    \end{aligned} 
 \end{equation}
 Here $V_0\ge0$, $\Gamma$ is the Gamma function, $W$ is a standard one-dimensional Brownian motion and $\lambda,\theta$ and $\sigma$ are given positive parameters.  Notice that, in the limiting case $\alpha=1$, \eqref{SVE_cont} reduces to the classical CIR process,  see \cite{CIR}. Weak existence results for \eqref{SVE_cont} can be found in, e.g., \cite{ALP19}.
 
 The process $V$ in \eqref{SVE_cont} has trajectories that are rougher than those of Brownian motion; more precisely, they are $(H-\varepsilon)-$H\"older continuous, as is the case for fractional Brownian motion (see, e.g., Theorem 3.2 in \cite{JR16}). Here $H=\alpha-\frac{1}{2}$ is known as the Hurst index. In the last decade, $V$ has been extensively investigated, especially in the mathematical finance literature, as it describes the variance in the so-called rough Heston stochastic volatility model; see \cite{ER18,ER19}. This model was introduced in \cite{JR16}, where the authors show, using $C$-tightness arguments, that it arises naturally as weak limit of a small jump high-frequency market microstructure model driven by two nearly unstable Hawkes processes. 
One of the main reasons for the popularity of this model is its affine nature; see \cite{ALP19,ER19,GK19}. This property allows to express the characteristic function of the log stock price and the integrated variance in rough Heston-type models in terms of solutions to deterministic non-linear Volterra integral equations (VIEs); see also \cite{A21,BLP24,BPS24,CT20} for extensions allowing jumps in $V$. This formulation enables accurate option pricing even for very small values of $H$ close to  0 (and even for $H=0$) by solving VIEs numerically via an Adams scheme. In particular, it avoids Monte Carlo techniques, which are known to perform poorly when $H\ll 1$, both in terms of bias and sample variance.\\ 

\sk 
In this paper, we study the integrated variance process $A=(A_t)_{t\ge0}$, defined by $A_t=\int_0^t V_s\,ds$. Our goal is to determine weak scaling limits for  $A$ and to characterize these limits as the laws of certain time-changed L\'evy processes. The results we obtain extend the existing literature on weak limits for the integrated variance in rough Heston-type models. They also apply to extensions of the dynamics in \eqref{SVE_cont} that exhibit jumps, in the spirit of \cite{A21,BLP24,BPS24}, for which -- to the best of our knowledge -- no similar results have been established so far. Specifically, we consider dynamics that include
\begin{equation}\label{SVE_j_intro}
   V_t =  V_0 \,+\,
	\int_0^t K(t-s) \Big(\lm(\theta-V^{}_s)ds+\sigma\sqrt{V_s} dW_s + d\tilde{J}_s\Big), 
 \end{equation}
where $K$ is the same fractional kernel as in \eqref{SVE_cont} and $\tilde{J}_\cdot=\int_0^{(\cdot)}\int_{\mathbb{R}_+} x (N^{}(dx,dt)-V^{}_t \nu(dx)dt) $. Here, $N(dx,dt)$ is an integer-valued random measure with compensator  $V_t^{}\nu(dx) dt$, for a  given non-negative measure $\nu$ with positive support s.t. $\nu(\lb 0\rb)=0$ and $\int_{\mathbb{R}_+}  x^2 \nu(dx)<\infty$. Weak existence results for \eqref{SVE_j_intro} are established in \cite{A21}.

In the continuous setting (see \eqref{SVE_cont}), significant contributions in this research direction can be found in \cite{AAR25,AC24}.
More precisely, 
\cite{AC24} show that a re-scaled standard Markov Heston model with fast mean-reversion and large vol-of-vol
(via an $H$ parameter which is not the Hurst exponent) tends weakly on path space to one of three different models (either Black-Scholes, a Normal Inverse Gaussian or a Normal L\'{e}vy model),
depending on whether their $H$ parameter is $>$, $=$, or $<$ $-\half$.
\cite{AAR25} obtain a similar result without Markovian approximations  but instead letting $H\searrow -\half$ for the so-called \textit{hyper-rough Heston model} (see also Section 5 in \cite{FGS21} and Section 7 in \cite{A21} for more on this model),
and exploiting Dirac-type behaviour in their Lemma 2.4 (see Appendix \ref{section:AAR25} here for a short summary/formal derivation of their result).\\
One of the contributions of this work is to fill the gap between \cite{AC24} and \cite{AAR25} by showing that a similar result holds for any $H\in(0,\tfrac12)$ (in particular, our limit regime with $H=\tfrac12$ corresponds to the regime in Eq.\,(0.3) of \cite{AC24}, where their $H=-\tfrac12$). This will follow as a consequence of a more general result presented in Section \ref{sec_jumps} (see Theorem \ref{cor_functional}), where we establish weak convergence on path space for an extended model that allows positive jumps in the dynamics of $V$, see \eqref{SVE_j_intro}. In this case, relying on the affine structure of the model, using the Laplace transform of the hitting time to a lower barrier for a spectrally positive L\'{e}vy process, we find that the limiting process for the integrated variance is a deterministic time-change of the first passage time process to lower barriers  for a more general class of L\'{e}vy processes. The convergence is proved using a compactness argument with the Kolmogorov-Riesz-Fr\'{e}chet theorem that is detailed in Appendix \ref{section:Conv}.
 
%and the variance curve $\xi_t(u)=\Ex(V_u|\mc{F}_t)$ evolves as $d\xi_u(t)=\kappa(u-t)\sqrt{V_t} dW_t$, where $\kappa(t)$ is the usual fractional kernel $t^{H-\half}$ for the $V$ process multiplied by a \textit{Mittag-Leffler} function. 
% The variance process $V$ is $(H-\e)-$H\"{o}lder continuous like fBM (see e.g. Theorem 3.2 in \cite{JR16}) and the spot and VIX smiles exhibit power-law skew for implied volatility in the small-time limit (see \cite{FGS21}, \cite{FGS22}, \cite{FSV21}  for details).  

%In related work, \cite{Jai15} considers two identical i.i.d. Hawkes processes $N^{\pm}$ with kernel $\phi \in L^1(0,\infty)$ and, assuming that the asset price $P_t=\text{const}\times \lim_{u\to \infty}\Ex[N^+_u-N^-_u|\mc{F}_t]$, i.e. a constant times the conditional expected value of all future order flow, then $P$ is clearly a martingale, but can also be written in the 
%propagator form $
%P_t= \int_0^t  \zeta(t-u) (dN^+_u-d N^-_u)\label{eq:prop}.
%$
%If
%$
%\zeta(t) =\kappa \sigma(1 - \int_0^t \phi(s) ds)\searrow \kappa \sigma(1-\|\phi\|)$ as $t\to \infty\nn
%$, so
$
%\zeta'(t) =-\zeta(0) \phi(t) \,\nn
$
%and if $\zeta(\infty)>0$, $\zeta(\infty)$ is the (non-transient) permanent price impact component of $\zeta$, and we can compute the expected market impact of an exogenous metaorder executed at constant rate $v$ over duration $\tau$ as $MI(t)=v\int_0^{t\wedge \tau}\zeta(t-s) ds$
%(see also \cite{JR20} for more on this).

 %\cite{BL24} improve accuracy in estimating the tail part of the integral for Fourier inversion (for e.g. call options under the rough Heston model) using a sinh contour which goes outside the usual strip of analyticity for the mgf (but is admissible as long as it avoid poles of the characteristic function), and also discuss refinements to the usual Adams schemes for solving the rough Heston VIE.  For articles on statistical estimation of $H$ in more general setups, we refer the reader to
 %\cite{CD24}, \cite{HS21}, \cite{BF18} (and follow-up papers on the LAN property and minimax theorems), and the asymptotic normality result for $\sqrt{n}(\hat{H}_n-H),\frac{\sqrt{n}}{\log n}(\hat{\sigma}_n-\sigma)$ in Theorem 1 in \cite{Syz23} for $\sigma B^H$ (where $B^H$ is fBM with $\sigma$ unknown) for the Whittle approximation $(\hat{H}_n,\hat{\sigma}_n)$ for the MLEs in terms of the Fisher information matrix in the high-frequency regime. 
 %where the usual $\sqrt{n}$ multiplier is replaced by $\sqrt{n}/\log(n)$.

\sk
Although we establish and focus here only on theoretical results, they may prove useful for applications that will be investigated in future research. For instance, in the continuous case \eqref{SVE_cont},
by the stochastic Fubini and Dubins–Schwarz theorems, it is well known that the integrated variance process satisfies an equation of the form
\[
A_t = G_0(t) + \int_0^t \kappa(t-s) B_{A_s}\, ds,
\]
for some Brownian motion $B$, where the deterministic curve $G_0$ and the kernel $\kappa$ are related to the coefficients in \eqref{SVE_cont} (see, e.g., \cite{A21, AAR25} and Theorem 3.1 in \cite{JR20}). This is a non-linear pathwise Volterra integral equation (VIE) for $A$ in terms of the a.s.\ $(\tfrac12-\varepsilon)-$H\"older continuous function $B_{(\cdot)}$, and it remains well defined even if $\kappa\in L^1(0,T)$; this also allows us to consider the so-called hyper-rough regime $H \in (-\tfrac12,0]$.  
If we discretize this equation for $A$ and rewrite it in terms of the final (discrete-time) increments of $A$ and $B$ (see Appendix \ref{section:AAR25} for details), then an independent sequence of Inverse Gaussian random variables can be used to perform an approximate Monte Carlo simulation of $A$ (see Algorithm 1 in \cite{AA25}), which is naturally suited for the regimes $H \ll 1$ and $H\in(-\tfrac12,0]$.

\sk
The paper is organized as follows. In Section \ref{asy_marginal}, we introduce a scaling for the continuous dynamics in \eqref{SVE_cont} to prove the convergence of the time-$t$ marginals of the integrated variance to an Inverse Gaussian process. This scaling is then extended in Section \ref{sec_jumps} to affine SVEs with jumps that include \eqref{SVE_j_intro} (see also Remark \ref{rem:2.1}). Theorem \ref{thm_main} establishes the convergence of the finite-dimensional distributions of the corresponding integral process to a L\'evy subordinator with a deterministic time shift. In Subsection \ref{sub_func}, this convergence result is strengthened to a functional setting. In particular, in Theorem \ref{cor_functional} we prove weak convergence of the laws of the re-scaled integral processes to the law of the aforementioned time-shifted L\'evy subordinator on the space of c\`adl\`ag functions on the non-negative half-line endowed with the $M_1$ topology (see \cite{W} for further details on this and related topological spaces, also used in \cite{AAR25,AC24}). In the appendices, we prove some results  needed for  our analysis of the jump case in Section \ref{sec_jumps}. More precisely, 
Appendix \ref{section:AlessE+U} studies the well-posedness of the VIE needed to describe the finite-dimensional moment generating function of the integrated variance. Appendix \ref{section:Conv} contains the proof of the main technical lemma of the paper, Lemma \ref{lem:Ale}, which, together with the results in Appendix \ref{section:LT}, enables us to identify the limiting law of the re-scaled integrated variance. Finally, Appendix \ref{section:AAR25} contains a formal derivation of the main idea in \cite{AAR25}. \\
Due to its interpretation as the variance process in  rough Heston-type  models, we will often refer to $V$ in \eqref{SVE_cont} (or, in the jump case, in \eqref{SVE_j_intro}; see Section \ref{sec_jumps}) simply as the ``variance process''.
% although the latter may give quite different results to the exact model for $H$ close to zero.
% where the time-change is governed by the asymptotic initial variance curve.  
%, Pannier
%FPS, probability space, look at barrier options, Wiener-Hopf

\section{Marginal Integrated Variance Asymptotics in a Re-Scaled Rough Heston Model}\label{asy_marginal}

\sk
We consider processes defined on (possibly different) probability spaces $(\Omega, \mathcal{F}, \mathbb{Q})$, each equipped with a filtration $(\mathcal{F}_t)_{t \ge 0}$ satisfying the usual conditions. 
For the main convergence result of this section, see Theorem \ref{cont_prop} below, we recall the following definition.
\begin{defn}
A random variable $X$ is said to have an \emph{Inverse Gaussian (IG)} distribution with parameters $\gamma > 0$ and $\delta > 0$, denoted by
$X \sim \mathrm{IG}(\gamma, \delta)$,
if its probability density function is
\[
f_X(x) = \bigg(\frac{\delta}{2\pi x^3}\bigg)^{1/2} 
\exp\bigg\{-\frac{\delta (x-\gamma)^2}{2 \gamma^2 x}\bigg\}, \quad x>0.
\]
\end{defn}
The moment generating function of $X \sim \mathrm{IG}(\gamma,\delta)$ is given by
\begin{equation}\label{mgf_IG}
 \mathbb{E}[e^{p X}] 
= \exp\bigg\{\frac{\delta}{\gamma} \bigg(1 - \sqrt{1 - \frac{2 \gamma^2 p}{\delta}} \bigg) \bigg\}, 
\quad p < \frac{\delta}{2 \gamma^2},
\end{equation}
which shows that the $\mathrm{IG}(\gamma,\delta)$ distribution is infinitely divisible. An IG process with parameters $(\gamma,\delta)$ is a (c\`adl\`ag) L\'evy process $L = (L_t)_{t \ge 0}$ such that $L_1 \sim \mathrm{IG}(\gamma,\delta)$. This process exists and is unique in law; see Theorem 7.10 in \cite{Sato99}. As is well known (see, for instance, Example 1.3.21 in \cite{A}), an $\mathrm{IG}(\gamma,\delta)$ process is given by 
\begin{equation}\label{constr_IG}
L_t=\inf\bigg\{s\ge 0 : B_s+ \frac{\sqrt{\delta}}{\gamma}s= \sqrt{\delta}t \bigg\},
\end{equation}
where $B$ is a standard one-dimensional Brownian motion.

In the next result, Theorem \ref{cont_prop}, we show that the time-$t$ marginals of the integrated solution to a re-scaled version of \eqref{SVE_cont} converge to those of an IG L\'evy process. This result is of interest as it demonstrates that such an IG process can be obtained without requiring $H=\alpha-\frac12$, where $\alpha$ denotes the exponent in the kernel $K$ driving \eqref{SVE_cont}, to be equal to $\frac12$ as in \cite{AC24}, or to tend to $-\frac12$ as in \cite{AAR25}. Therefore, it extends the existing literature on scaling limits of rough Heston models.\\
The proof relies on the affine structure of \eqref{SVE_cont}, which enables us to express the Laplace transform of the integrated variance  via deterministic Riccati--Volterra equations. This is a well-known property of affine Volterra processes; see, e.g., \cite{ALP19}, which generalizes the corresponding classical results for standard affine processes, for instance those in \cite{DFS}. Thanks to a suitable rescaling, in the proof of Theorem \ref{cont_prop} we show that the moment generating function of the integrated variance  converges to that of an IG process, which implies convergence of the time-$t$ marginals.
\\
The rescaling introduced in Theorem \ref{cont_prop} will then be generalized in Section \ref{sec_jumps} to the case with jumps. In particular, as a corollary of Theorems \ref{thm_main} and \ref{cor_functional}, the convergence in Theorem \ref{cont_prop} will be strengthened to a functional setting; see Remark \ref{cor_continuous}.
\begin{thm}\label{cont_prop}
	Consider a re-scaled rough Heston model for the variance process $V^\e=(V^\e_t)_{t\ge0}$:
	\bq
	V^{\e}_t &=&  V_0 \,+\,
	\frac{1}{\Gamma(\al)}\int_0^t (t-s)^{\al-1} \Big(\frac{1}{\e}\lm(\theta-V^{\e}_s)ds+\frac{1}{\e}\sigma\sqrt{V^{\e}_s} dW_s\Big),\label{eq:Org}
	\eq
	where $V_0\ge0$,  $W$ is a standard one-dimensional Brownian motion, $\al\in (\half,1)$
	and $\lm,\sigma>0$.
	Then, for every $t>0$ fixed, $A^\e_t=\int_0^tV^{\e}_s ds $ tends weakly,  as $\e\to 0$, to the time-$t$ marginal of an Inverse Gaussian L\'{e}vy process with parameters $(\theta,\sigma^{-2}\lambda^2\theta^2)$
	which does not depend on $H=\al-\half$.
\end{thm}
 \begin{proof}
First, we notice that weak existence of a non-negative solution to \eqref{eq:Org} is established in, for instance, Theorem 7.1 of \cite{ALP19}.	%see Proposition 4.2 in \cite{FGS21}.
Let $I^{\al}(f)(t)=\frac{1}{\Gamma(\al)}\int_0^t (t-s)^{\al-1}f(s)ds$ denote the $\al$th-order fractional  integral of a function $f$ for $\al \in (\frac12,1)$. Then for $p> 0$
	(which will be sufficient for our purposes when we invoke a classical weak convergence result from \cite{Bill86} below), if we define $A^\e_t=\int_0^tV^\e_s ds,$ by the arguments in, e.g., Section 7 in \cite{ALP19}
	% in some interval $I=(p_-(t),p_+(t))$ with $[0,1]\subset I$, we know that
	\bq
	\Ex[e^{-p  A^{\e}_t}]&=&e^{V_0 I^{1-\al}\phi_{\e}(t)+\frac{1}{\e}\lm \theta I^1 \phi_{\e}(t)},\label{eq:CF}
	\eq 
	%for $p$ in some open interval $I \supset [0,1]$, 
	where $\phi_{\e}$ is the unique continuous solution of the non-linear Volterra integral equation (VIE):
	\bq
	\phi_{\e}(t) &=& \frac{1}{\Gamma(\al)} \int_0^t (t-s)^{\al-1}\Big(-p\,-\,  \frac{1}{\e}\lm\phi_{\e}(s) \,+\,  \frac{1}{\e^2}\half \sigma^2\phi_{\e}(s)^2\Big)ds.\nn \,
	\eq In particular, note that $\phi_\e$ also depends on $p$. 
    
    Now let $\phi_{\e}(t)=\e \psi(\e^q t)$.  Then, also using the change of variables $\e^qs=u$,
	\bq
	\e \psi(\e^q t) &=& \frac{1}{\Gamma(\al)} \int_0^t (t-s)^{\al-1}\Big(-p\,-\, \lm \psi(\e^q s)\,+\,  \half \sigma^2\psi(\e^q s)^2\Big)ds\nn\\
	%&=& \frac{1}{\Gamma(\al)} \int_0^{\e^q t} (t-u \e^{-q})^{\al-1}\Big(-p\,-\,  \lm \psi(u)\,+\,  \half \sigma^2\psi(u)^2\Big)du \,\e^{-q}\nn\\
	&=& \frac{\e^{-q (\al-1)}}{\Gamma(\al)} \int_0^{\e^q t} (\e^q t -u )^{\al-1}\Big(-p\,-\,  \lm \psi(u)\,+\,  \half \sigma^2\psi(u)^2\Big)du \,\e^{-q}.\nn\,
	\eq
Setting $\e^q t \mapsto t$, we see that
	\bq
	\e \psi(t) 
	&=& \frac{\e^{-q \al}}{\Gamma(\al)} \int_0^t (t -u )^{\al-1}F(\psi(u))du, \,\label{eq:VI}
	\eq
	where $F(w)=-p-\lambda w+  \half \sigma^2 w^2$. If now we let $q=-\frac{1}{\al}$, the VIE \eqref{eq:VI} is independent of $\e$, so 
    \begin{equation}\label{agg_noeps}
        \phi_{\e}(t)=\e \psi\bigg(\frac{t}{\e^{\frac{1}{\al}}} \bigg).
    \end{equation}
	Hence for every $t>0$, from Lemma 4.5 in \cite{FGS21}
	%since $(p_-,p_+)\subseteq (p_-(t),p_+(t))$ we see that
	\bq\label{Lemma_4.5}
	\lim_{\e\to\color{black}0}\frac{1}{\e}\phi_{\e}(t) &=& \lim_{t\to\infty}\psi(t) ~\psi(\infty) ~ \frac{1}{\sigma^2} \Big[\lm - \sqrt{\lm^2 + 2 p\sigma^2}\Big]
    \,\,\,\,=:\,\,\,\,
    U_1(p)
	\eq
	for $p >0$. %	and the convergence to $\psi(\infty)$ is rapid for $\e\ll 1$ since $\frac{1}{\e}\phi_{\e}(t)= \psi\big(\frac{t}{\e^{\frac{1}{\al}}} \big)$. 
    More precisely, Lemma 4.5 in \cite{FGS21} can be applied to  $-\psi$, which solves the following VIE on $\mathbb{R}_+$:
	\[
		-\psi(t)= \frac{1}{\Gamma(\alpha)}\int_{0}^{t}(t-u)^{\alpha-1}F_1(-\psi(u))du,\quad \text{with  }F_1(w)=p-\lm w-  \half \sigma^2 w^2.
	\] 
    Indeed, since $F_1(0)>0$ and $F_1$ is analytic and decreasing on $\mathbb{R}_+$, the aforementioned result in \cite{FGS21} implies that $-\psi(t)$ is monotonically increasing and 
 converges as $t\to \infty$ to the positive root of $F_1$, which coincides with $-U_1(p)$.
	Then for the exponent in \eqref{eq:CF}, knowing from \eqref{agg_noeps}, \eqref{Lemma_4.5} and the monotonicity of $\psi$ that $|\frac{1}{\e}\phi_\e(t)|\le -U_1(p)$ for every $t\ge0$ and $\varepsilon>0$, by the dominated convergence  theorem and \eqref{Lemma_4.5} we see that
	\bq
	V_0 I^{1-\al}\phi_{\e}(t)\,+\, 
	\frac{\lm \theta}{\e}I^1\phi_{\e}(t) =
	\frac{V_0}{\Gamma(1-\al)} \int_0^t (t-s)^{-\al} \phi_{\e}(s)ds \,+\, \lm \theta\int_0^t \frac{1}{\e}\phi_{\e}(s)ds
	&\to& 0\,+\, \lm \theta U_1(p) t \nn
	\eq
	as $\e\to 0$. By \eqref{mgf_IG},
	% since the first term here is $O(t^{\al})$ as $t\to \infty$,
	 $\lm \theta U_1(p)t$ is the log moment generating function of an Inverse Gaussian L\'{e}vy process with parameters $(\theta,\sigma^{-2}\lambda^2\theta^2)$ at time $t$.
	Then from e.g. the solution to Problem 30.4 on Page 573 in  \cite{Bill86}, we conclude that 
	$A^{\e}_t$ tends weakly to the time-$t$ marginal law of an $\mathrm{IG}(\theta,\sigma^{-2}\lambda^2\theta^2 )$ process.
\end{proof}
\sk

\section{Adding jumps into $V^{\e}$}\label{sec_jumps}

\sk
We now assume that the forward variance $\xi^{\e}_t(u):=\Ex[V^{\e}_u|\mc{F}_t]$ satisfies
\bq
d\xi^{\e}_t(u) &=& \kappa_{\e}(u-t)(\sigma \sqrt{V^{\e}_t} dW_t\,+\,d\tilde{J}^{\e}_t),\quad {\color{black}u>t,}\label{eq:xi}
\eq
where $d\tilde{J}^{\e}_t=\int_{\mathbb{R}_+} x (N^{\e}(dx,dt)-V^{\e}_t \nu(dx)dt) $ and $N^{\e}(dx,dt)$ is a (time-inhomogenous) Poisson random measure with (random) intensity $V_t^{\e}\nu(dx) dt$; $\nu$ only has positive support with $\nu(\lb 0\rb)=0$ and $\int_{\mathbb{R}_+}  x^2 \nu(dx)<\infty$, so $\tilde{J}^{\e}$ has positive-only jumps. The kernel $\kappa_{\e}$ is defined by $$\kappa_{\e}(t)= \frac{1}{\e}t^{\al-1}E_{\al,\al}\Big(-\frac{\lm}{\e} t^{\al}\Big),\quad \text{with $\al \in \Big(\half,1\Big)$ and $\lambda>0$,}$$  where $E_{\al,\beta}(z)$ denotes the Mittag-Leffler function. We refer to Remark \ref{rem:2.1} below for relevant examples from the literature (e.g., \cite{BLP24,BPS24, CT20,Cuch22}) that fall within this setting, including the continuous dynamics in \eqref{eq:Org} (where $\nu(dx)\equiv0$).
% (see e.g. Appendix A.1 of \cite{ER19} for definition).  

\sk
The critical observation for the arguments that follow is that $$f^{\al,\lm}(t)= \lm t^{\al-1}E_{\al,\al}(-\lm t^{\al})$$
is a probability density, so  $\lm \kappa_{\e}(\cdot)$ has Dirac-type behaviour as $\e\to 0$.
We also mention that
$
	\int_{t}^{\infty}f^{\alpha,\lambda}(s)\, ds\underset{t\to\infty}{\sim}\frac{1}{\lambda\Gamma(1-\alpha)}t^{-\alpha},
$
which implies that
\begin{equation}\label{asymptotic}
	\int_{t}^{\infty}\lambda \kappa_\e (s)ds \underset{\epsilon\to0}{\sim} \frac{1}{\lambda\Gamma(1-\alpha)}\e t^{-\alpha}
\end{equation}
for $t>0$ (see Appendix A.1 of \cite{ER19} for details on these points).

\sk

The variance process $V^{\e}=(V^{\e}_t)_{t\ge0}$ appearing in \eqref{eq:xi} is predictable, non-negative, has trajectories in $L^1_\text{loc}(\mathbb{R}_+)$ (see e.g. \cite{A21}, \cite{ACLP} and \cite{BLP24}) and satisfies the following SVE of {convolution-type} with jumps: 
	\begin{equation}\label{SVE_jumps}
	V^{\e}_t =  \xi_0^{\e}(t) \,+\,
	\int_0^t \kappa_{\e}(t-s) (\sigma\sqrt{V^{\e}_s} dW_s+d\tilde{J}^{\e}_s)\,,\quad \mathbb{P}\otimes dt-\text{a.e.},
	\end{equation}
	where $\xi_0^\e\in L^1_\text{loc}(\mathbb{R}_+)$ is the initial variance curve. 
	In particular, $V^\e_t=\xi^\e_t(t),\,\mathbb{P}-$a.s., for a.e. $t\in (0,\infty)$. Note that we require $V^\e$ to be non-negative in order to consider the square root in \eqref{SVE_jumps}, while the predictability of $V^\e$   with locally integrable paths ensures that  the stochastic integrals in \eqref{SVE_jumps} are properly defined. We refer to \cite{A21} for weak existence results for \eqref{SVE_jumps}, see also \cite{ACLP}.

	The process $V^\e$ here is an affine Volterra process with jumps and falls under the framework of \cite{BLP24} (see also \cite{BPS24}). 
%	{\color{cyan}\\Ale: I wouldn't specify a part of the paper, in the sense that \cite{BLP24} deals with general affine processes with jumps covering also the dynamics in \eqref{SVE_jumps}. But let me know if I misunderstood.
%		\\}
	In particular, Lemma 1 in \cite{BLP24} establishes the following integrability property of $V^\e$ that we will use for our analysis:
	\begin{equation}\label{L1integrability}
\mathbb{E}\bigg[\int_{0}^{T} V^\e_t\,dt
\bigg]		<\infty,\quad T>0. 
	\end{equation} 
We also refer to Lemma 12 in \cite{BLP24} for a stronger $L^2-$type integrability result which applies to our dynamics in \eqref{SVE_jumps} when $\xi^\e_0\in L^2_\text{loc}(\mathbb{R}_+)$.  {From  \eqref{L1integrability}, we see that}
\[
	\mathbb{E}\bigg[\int_{0}^{T}\bigg(\int_{\mathbb{R}_+}|x|^2\nu(dx)\bigg)V^\e_tdt\bigg]<\infty,\quad T>0,
\]
{and} hence 
\begin{align}\label{tildeJ_mg}
\notag	&\tilde{J}^\e= \int_{0}^{\cdot}\int_{\mathbb{R}_+}x (N^{\e}(dx,dt)-V^{\e}_t \nu(dx)dt ) \text{ is a square-integrable martingale in }[0,T],\\&\qquad \text{ for every }T>0.
\end{align}
% The $V_t=\xi_t(t)$ process here falls under the framework of Eq 1 in \cite{BPS24}.
% so $\tilde{J}^{\e}$ is a square integrable martingale.
%Note if $\sigma=0$ and $$

%and (as \cite{BPS24} remark).we can think of $V$ as a Hawkes-type process since the jump intensity is proportional to $V_t$ itself.
%Although we do not explicitly consider the stock price process $S$ in this section,
%a proof of the martingale property for $S$ is given in Section 3 of \cite{BPS24}. 

\sk
  Note that, %we are now using $\sigma$, not $\nu$, for the vol-of-vol term in \eqref{eq:xi}, since $\nu$ is being used here for the L\'{e}vy density, and 
  in the absence of jumps, \eqref{eq:xi} is the usual equation for the forward variance under the standard rough Heston model,
see e.g. \cite{ER18} or Proposition 2.2 in \cite{FGS21}.  The model in \eqref{eq:xi} can be viewed as a generalized rough Heston model {\color{black}in the spirit of \cite{BPS24}},  where the mean-reversion speed, vol-of-vol, and jump-intensity all scale as $\frac{1}{\e}$.
\vspace{2mm}

For our asymptotic analysis, we require the following assumption.
\begin{Assumption}\label{ass1}
\label{Assumption:2.1}
%\bl{CHANGE/REMOVE THIS}.  
We assume that  $\xi^{\e}_0(\cdot)$ is non-negative, uniformly bounded and continuous and $\xi^{\e}_0(\cdot)$ tends pointwise to a bounded  continuous function
 $\xi^{0}_0(\cdot)$ as $\e\to 0$.
%(and independent of $\e$).  
\end{Assumption}
{\color{black}  In the following remark, we present two important and standard cases in which Assumption \ref{ass1} is satisfied.
}
\begin{rem}
\label{rem:2.1}
Consider the SVE
\bq
V^{\e}_t &=&  V_0 \,+\,
\int_0^t K(t-s)\frac{1}{\e} \Big(\lm(\theta-V^{\e}_s)ds+\sigma\sqrt{V^{\e}_s} dW_s \,+\,d\tilde{J}^{\e}_s\Big)\label{eq:V_alt}
 \eq
with $K(t)=\frac{1}{\Gamma(\alpha)}t^{\al-1}$ (with the same jump structure for $\tilde{J}^{\e}$ as in \eqref{eq:xi}). The weak existence of a solution to \eqref{eq:V_alt} is established in Theorem 2.13 of \cite{A21}. Notice that, in the continuous case  $\nu(dx)=0$ (hence $\tilde{J}^\e=0$), \eqref{eq:V_alt} reduces to \eqref{eq:Org} in Theorem \ref{cont_prop} of Section \ref{asy_marginal}.
\\The SVE \eqref{eq:V_alt} is a special case of Eq. (14) in \cite{BLP24}
with their $g_0\equiv V_0$, $b_0=\frac{1}{\e}\lm \theta$, $B=b_1=-\frac{\lm}{\e}$, $A_0=0$, $A_1=\frac{1}{\e^2}\sigma^2$, $\nu_0=0$, $\nu_1(dx)=\nu(\e dx)$. From the  equation after (A.3)  in \cite{BLP24},  this process is equivalent to \eqref{SVE_jumps}, for which their $R_B=\lm \kappa_{\e}$, %\footnote{$R_B$ is the resolvent of the second kind of $-K B$, see e.g. Page 13 of \cite{BLP24} for definition.} 
$E_B=K-R_B*K={\e}\kappa_{\e}$, and $\xi^{\e}_0(\cdot)=g_0-R_B*g_0+E_B*b_0=
V_0-(V_0-
\theta)\int_0^{\cdot} f^{\al,\frac{\lm}{\e}}(s) ds$ (which agrees with Proposition 2.1 in \cite{FGS21}). For  the corresponding forward variance processes, compare Remark 5 in \cite{BLP24} with \eqref{eq:xi}. In this case, from \eqref{asymptotic}, we find that $\xi^{\e}_0(u)\to \theta$ for $u>0$ as $\e\to 0$, which we can also easily obtain from Proposition 2.1 in \cite{FGS21} since $\tilde{J}^\e$
is a martingale, so the jumps do not affect $\xi_0^{\e}(t)=\Ex[V^{\e}_t]$. \vspace{2.5mm}

\noindent
Conversely, if $\xi_0$ in \eqref{SVE_jumps} is independent of $\e$ and given exogenously, then we can find a $g_0$ in \cite{BLP24} consistent with
$\xi_0$ by solving the linear VIE $g_0-R_B*g_0+E_B*b_0=\xi_0(\cdot)$. Specifically, letting 
$ f=\xi_0(\cdot) - E_B*b_0$, we can re-write the VIE as $g_0-R_B*g_0=f$, which has solution $g_0= f - f \ast r$. Here $r$ is the resolvent of the 2nd kind of $-R_B$
(which will depend on $\e$ in general); see Section 2.3 in \cite{GLS90}.
\end{rem}

%\footnote{an instructive example to keep in mind here is the case when $V_1$ is the cgf for a one-sided tempered stable (CGMY) process with $\nu(dx)= \frac{C e^{-M x}}{x^{1+Y}}1_{\{x>0\}}$ for $C,M>0$ and $Y\in (0,2)\setminus \lb 1\rb$, for which $V_1(p)=C( M (M - p)^Y + M^Y (-M + p Y))\Gamma(-Y))/M$.} 
To prove the main result of this section, Theorem \ref{thm_main} below, we rely on the affine structure of $V^\e$, which enables us to express the Laplace transform of its convolution with locally bounded functions via deterministic non-linear Riccati--Volterra equations. The following lemma establishes the well-posedness of the equations required in our analysis.
\begin{lem}
\label{lem:E+U}
Consider a non-positive locally bounded function $f\in L^\infty_\emph{loc}(\mathbb{R}_+;\mathbb{R}_-)$. Define  the map $G\colon \mathbb{R}_+\times \mathbb{R}_-\to\mathbb{R}$ by 
\bq
G(s,w)&=&f(s) +\half \sigma^2 w^2 +V_1(w),\quad \text{where }V_1(w) = \int_{\mathbb{R}_+} (e^{w x}-1- wx)\nu(dx),\label{eq:G}
\eq
for  $(s,w)\in \mathbb{R}_+\times \mathbb{R}_-.$
Then, for every $\e>0$, there exists a unique continuous non-positive solution $\psi_\e\colon \mathbb{R}_+\to \mathbb{R}_-$ to the Riccati--Volterra equation 
\bq\label{eq:Vint}
\psi_{\e}(t)
&=& \int_0^t \kappa_{\e}(t-s) G(s,\psi_{\e}(s))ds,\quad t\ge 0. \eq
\end{lem}
\begin{proof}
See Appendix \ref{section:AlessE+U}. We notice that $V_1$ in \eqref{eq:G} is well-defined, i.e., $V_1(w)<\infty$ for every $w\le 0$,  since $|e^{wx}-1-wx|\le w^2 |x|^2$ for any $x\in\mathbb{R}_+$.
\end{proof}
Formally, the asymptotic solution to \eqref{eq:Vint} comes from considering its Dirac limit as $\e \to 0$:
\bq\nn
\psi_{0}(t)
&=&  \frac{1}{\lm }G(t,\psi_{0}(t))  \eq
(recall from above that $\lm \kappa_{\e}(\cdot)$ has Dirac-type behaviour as $\e\to 0$, so $\lm \kappa_{\e}(t-s)$ will be concentrated
at $s=t$).
\sk
Re-arranging terms here, we obtain our conjecture limit equation:
	\begin{equation}\label{lim_eq}
		f(t)-\lambda \psi_{0}(t) + \bar{G}(\psi_0(t))=-\lambda \psi_{0}(t) +G(t,\psi_0(t))=0,\quad t\ge0,
\end{equation}
where $G$ is defined as in \eqref{eq:G} and $\bar{G}(w)= \frac{1}{2}\sigma^2w^2 + \int_{\mathbb{R}_+}(e^{xw}-1- xw)\nu(dx)$ for $w\le 0$. In the next lemma, we prove that \eqref{lim_eq} is well-posed. %Observing that the $\mathbb{R}_+-$valued map $w \mapsto -\lambda w + \bar{G}(w)$ is decreasing in $\mathbb{R}_-$ with $\lim_{w\to -\infty}\bar{G}(w)=\infty$, 
\begin{lem}\label{lem_aggiunto}
    For any $f\in L^\infty_\emph{loc}(\mathbb{R}_+;\mathbb{R}_-)$, there exists a unique  solution $\psi_0$ to \eqref{lim_eq}. Furthermore, $\psi_0\in L^\infty_\emph{loc}(\mathbb{R}_+;\mathbb{R}_-)$.
\end{lem}
\begin{proof}
    {Since, by the definition in \eqref{eq:G},  $G(t,0)=f(t)\le 0$ and the function $w\mapsto G(t,w)-\lambda w$ is {\color{black}continuous and decreasing} on $\mathbb{R}_-$, with $G(t,w)-\lambda w\to \infty$ as $w\to-\infty$},
we see that there exists a unique non-positive solution $\psi_{0}\colon\mathbb{R}_+\to \mathbb{R}_-$ to \eqref{lim_eq}. In particular, {\color{black}since $f$ is locally bounded on $\mathbb{R}_+$,} $\psi_0$ is locally bounded on $\mathbb{R}_+$, as well (i.e., $\psi_0\in L^\infty_{\text{loc}}(\mathbb{R_+};\mathbb{R_-})$).
%{\\\color{cyan}Ale: Here I prefer to say that the map is decreasing: we do not really need convexity (besides, convexity should be proved, while monotonicity is immediate)\\}
\end{proof}
We now provide the main technical lemma needed in our analysis. It clarifies in which sense the solution $\psi_0$ to \eqref{lim_eq} can be interpreted as the asymptotic solution to \eqref{eq:Vint}.
\begin{lem}
\label{lem:Ale}
Consider the solutions $(\psi_\e)_\e$ and $\psi$ to \eqref{eq:Vint} and \eqref{lim_eq}, respectively. Then, for every $T>0$,
\[
\lim_{\e\to 0}\psi_\e=\psi_0	\quad \mathrm{in\,\,}L^1(0,T).
\]
\end{lem}
\begin{proof}
See Appendix \ref{section:Conv}. %{\\	\color{cyan}Ale: I removed $\e\to 0+$ (and now simply write $\e\to 0 $), to unify the notation with the rest of the paper}
\end{proof}

\begin{rem}

%For the process $V^{\e}$ in \eqref{eq:V_alt}, if we instead set $M_t=e^{ \int_0^t f(T-s) V^{\e}_s ds + G_t}$ with
%$
%G_t =\int_t^{T} \tilde{g}_{\e}(T-s) g^{\e}_t(s) ds\nn\,
%$
%and we define the ``adjusted forward process'' $g^{\e}_t(s)$ as in Eq. (3) in \cite{BPS24} (see also Eq. (37) in \cite{BLP24}) by
%\bq
%g^{\e}_t(s) &=& V_0 \,+
%\frac{\lm\theta}{\e}\int_{0}^{s}K(r)dr
%\,+\,\int_0^t K(s-r) \frac{1}{\e}\Big(-\lm V^{\e}_r dr +\sigma \sqrt{V^{\e}_r}dW_r+d\tilde{J}^{\e}_r\Big)\nn
%\eq
%for $s> t$, so $g^{\e}_t(t)=V^{\e}_t$, $\mathbb{P}-$a.s., for a.e. $t>0$ (with $K(t)=\frac{1}{\Gamma(\al)}t^{\al-1}$), 
%and $g^{\e}_0(t)=V_0+\lm \theta \int_0^t K(t-s)  ds=V_0+\lm \theta \int_0^t K(s)  ds$, 
%\bq
%dg_t(s) &=& K(s-t) (-\lm V_t dt +\nu \sqrt{V_t}dW_t +d\tilde{J}_t)\nn\,
%\eq
%then following  arguments  analogous to those in Appendix \ref{section:VIEproof} we can check that $\psi_{\e}$ in \eqref{eq:Vint} satisfies
%
%\end{comment}
  A variation of constants argument, see Remark 5 in \cite{BLP24} and also Lemma 4.4 in \cite{ALP19}, shows the equivalence between \eqref{eq:Vint} and 
\begin{equation}\label{eq:alt}
\e \psi_{\e}(\tau)= \int_0^{\tau} K(\tau-s)\Big(f(s)-\lm \psi_{\e}(s) \,+\, \half \sigma^2 \psi_{\e}(s)^2+V_1(\psi_{\e}(s))\Big)ds,\quad \tau\ge 0,
\end{equation}
where $K(t)=\frac{1}{\Gamma(\alpha)}t^{\alpha-1}$ is the fractional kernel. 
  According to Lemma \ref{lem:Ale}, the limiting solution as $\e\to 0$ is  $\psi_{0}$
% in the proof of Proposition 2.2
(we test this numerically in Figure \ref{figure_VIE}).
\end{rem}

\begin{figure}
 \begin{center}
 \includegraphics[width=165pt,height=140pt]{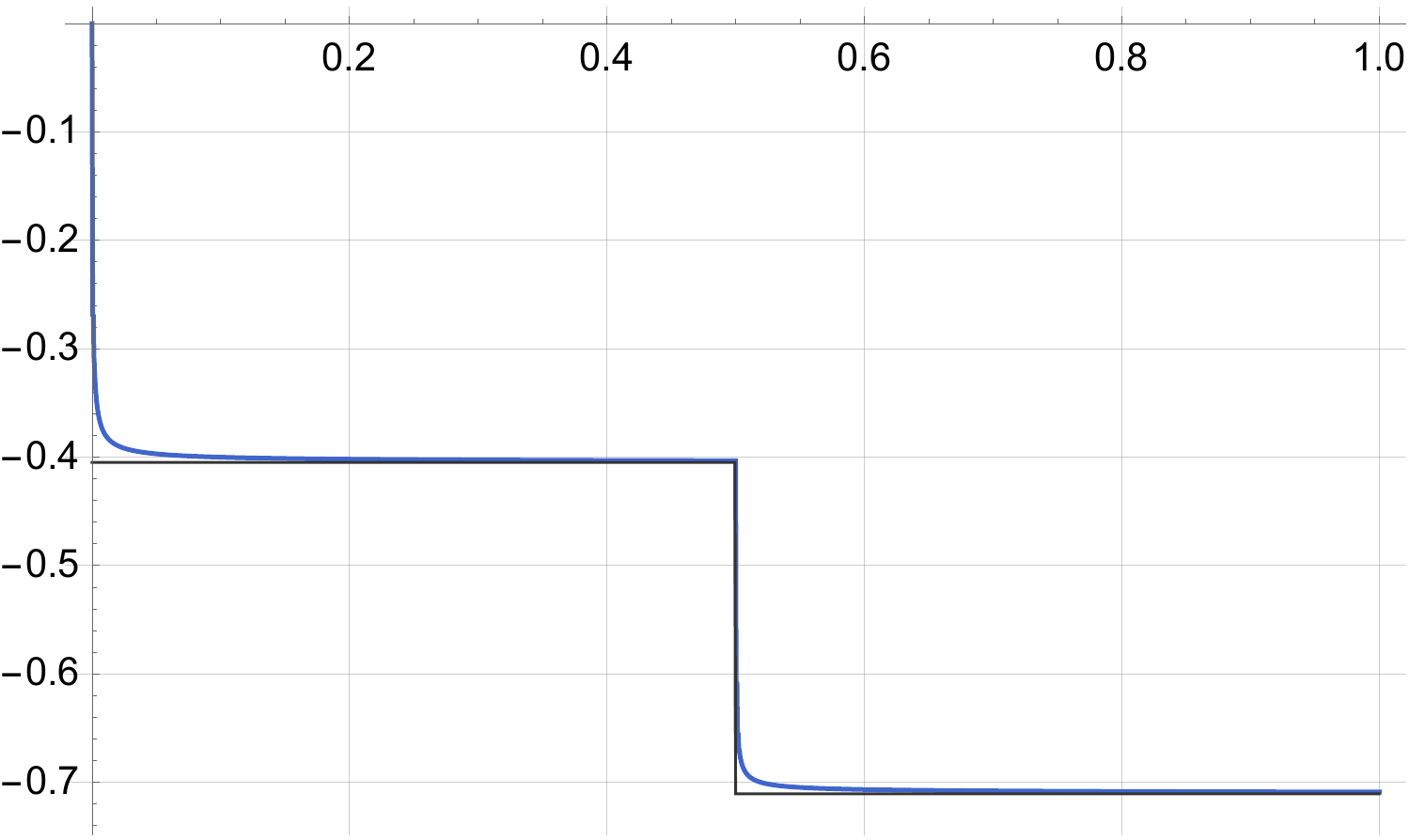} \quad 
\caption{Here we have plotted $\psi_{\e}$ in Eq. \eqref{eq:alt} (in blue) and $\psi_0$ in Eq. \eqref{lim_eq} (grey dashed), using an Adams scheme with 2000 time steps with 
$\e= .01$, $H =0.2$, $\nu= .4$, $\lm = 1$, $T = 1$, $f(s) =-\half (1_{\{s<\half\}} +1_{\{s\le 1\}})$
and $\nu(x)=\frac{C e^{-M x}}{x^{1+Y}}1_{\{x>0\}}$ for $C=1$, $M=3$ and $Y=1.5$. Notice that such a $\nu$ is the L\'evy measure of a one-sided tempered stable (CGMY) process. We see convergence to $\psi_0$ (see e.g. \cite{BL24} for details on refinements to Adams schemes).  Numerically solving the VIE in Eq. \eqref{eq:Vint}
for $\e\ll 1$ appears to be much harder due to the Dirac nature of the kernel.\label{figure_VIE}
}
   \end{center}
 \end{figure}

We now state the main theorem of this section. We note that this result, which concerns the convergence of the finite-dimensional distributions of the integrated process $V^\e$, will be strengthened in Theorem \ref{cor_functional} of Subsection \ref{sub_func}, where weak convergence in the path space of c\`adl\`ag functions endowed with the $M_1$ topology is established. As in the proof of Theorem \ref{cont_prop} in Section \ref{asy_marginal}, the idea behind Theorem \ref{thm_main} consists in exploiting Lemma \ref{lem:Ale} to study the limit of the moment generating function of the finite-dimensional distributions of the re-scaled integrated variance process.
\begin{thm}\label{thm_main} The finite-dimensional distributions of $A^{\e}_{\cdot}=\int_0^{(\cdot)}V^{\e}_s ds$ tend weakly to those of
a time-changed L\'{e}vy process $X_{g(\cdot)}$, where $X_t=\inf \lb s: Z_s< -t\rb$, $Z=(Z_t)_{t\ge0}$ is a L\'{e}vy process with L\'{e}vy triple $(-\lm,\sigma^2,\nu)$, and $g(t)=\lm \int_0^t \xi^0_0(u)du$.
% and $X$ is a L\'{e}vy subordinator.
\end{thm}
\begin{rem}\label{rem_subordinator}
	$X$ is a L\'{e}vy subordinator, see e.g. Theorem 46.2 in \cite{Sato99}; in particular, if we let $(\tilde{\lambda},0,\tilde{\nu})$ denote the L\'evy triple of $X$ with respect to the truncation function $h(x)\equiv0$, then $X$ has no Gaussian component, $\tilde{\nu}(\mathbb{R}_-)=0$ and $\int_{[0,1]}x\tilde{\nu}(dx)<\infty$, see e.g. Theorem 21.5 in \cite{Sato99}. 
	Moreover, if $\tilde{\nu}(dx)$ has a  density with respect to Lebesgue measure, i.e. $\tilde\nu(dx)=\tilde{\nu}(x)dx$, then the L\'{e}vy Khintchine-type formula in Theorem 25.17 in \cite{Sato99} for $X_T$ is:
	\[
	\Ex[e^{-p X_T}] = e^{T(-{\tilde{\lambda}}p + \int_0^\infty (e^{-px}-1)\tilde{\nu}(x)dx)}\nn
	\]
	for $p\ge 0$.
	Thus $\tilde{V}(-p): = \log\Ex[e^{-p X_1}] = -\tilde{\lambda}p + \int_0^\infty (e^{-px}-1)\tilde{\nu}(x)dx$, and differentiating both sides with respect to $p$, we obtain:
	\bq
	\tilde{V}'(-p) &=& \tilde\lambda + \,\int_0^\infty x{e^{-px}}\tilde{\nu}(x)dx\nn\
	\eq
	for $p\ge 0$, so in principle we can recover $\tilde\nu$ from $\tilde{V}$ by Laplace inversion.
\end{rem}
\begin{proof}
Recall that the dynamics of the variance process $V^\e$ are given in \eqref{SVE_jumps}.
Let $f:[0,\infty)\to (-\infty,0]$ be 
%a bounded measurable function
a locally bounded function on $\mathbb{R}_+$ (i.e., $f\in L^\infty_{\text{loc}}(\mathbb{R}_+;\mathbb{R}_-)$).  
% \footnote{the strict negativity here will be needed in Appendix A.}.
  Then from Theorem 5 in \cite{BLP24} and Lemma 6.1 in \cite{A21} (see also Lemma 3.2 in  \cite{BPS24}), we know that for every $T>0$ % (\textbf{real/complex argument}, $p$ in certain range)
%if $\|f\|_{L^{\infty}}$ is sufficiently small, then
\bq
M_t &:=&e^{\int_0^t f(T-s) V^{\e}_s ds + G_t}\label{eq:CF_},\quad t\in[0,T]
\eq is a martingale if we set 
 \bq
G_t ~ 
\int_0^{T-t}G(u,\psi_{\e}(u))\,\xi^{\e}_t(T-u) du,
\eq
where $G$ is defined in \eqref{eq:G} and $\psi_\e$ 
 satisfies the corresponding non-linear Riccati--Volterra integral equation \eqref{eq:Vint}. 
Existence and uniqueness for $\psi_\e$ are established in Lemma \ref{lem:E+U} above.  %{\color{cyan}\\Ale :  in Appendix A we now prove directly the true martingale property. We do not claim anything about the sign of $G(\cdot,\psi_\e(\cdot))$ anymore, hence  we cannot justify the previous upper bound. I also moved there the footnote about \cite{GK19}\\}
%{\color{red}Since $V^\e$ is nonnegative, $f$ is bounded and nonpositive, and $\psi_{\e}$ is continuous on $\mathbb{R}_+$, for every $T>0$ there exists a constant $C_T$ such that 
%}
%\bq
%M_t &\le&e^{\int_0^t f(T-s) V^{\e}_s ds} \lee 1 \,.\nn
%\eq
%Thus $M$ is a bounded local martingale, and hence a martingale 
 Since $G_T=0$ we see that
\bq
M_t &=&\Ex[M_T|\mc{F}_t] ~ \Ex\Big[e^{\int_0^T f(T-s) V^{\e}_s ds}|\mc{F}_t\Big]. \nn
\eq
% and we assume $V_1(p)$ is strictly sublinear as $p\searrow-\infty$
% $\lim_{u\searrow \infty}(-\lm u +V_1(u))=\infty$ is strictly sublinear as $p\searrow -\infty$. 
In particular, at $t=0$, taking the expected value we have
\bq
\Ex\Big[e^{\int_0^T f(T-s) V^{\e}_s ds}\Big] ~
%~ e^{\int_0^T g_{\e}(T-s)\xi_0(s) ds} ~ e^{\int_0^T g_{\e}(s)\xi_0(T-s) ds} ~
e^{\int_0^T  G(s,\psi_{\e}(s))\xi^{\e}_0(T-s) ds}.\label{eq:mgf}
\eq

As a corollary of Lemma \ref{lem:Ale}, we see that
\begin{align}\notag
\notag&\int_0^T |G(s,\psi_\e(s))-G(s,\psi_0(s))| ds 
\\\notag&\qquad\le\int_{0}^{T} \Big(
\frac{1}{2}\sigma^2|\psi_\e(s)+\psi_0(s)|+|\tilde{h}(\psi_{\e}(s),\psi_0(s))|
\Big)
|\psi_\e(s)-\psi_0(s)|ds
\\
&\qquad
\le 
K_1 \int_0^T |\psi_\e(s)-\psi_0(s)| ds \underset{\e\to 0}{\longrightarrow} 0 \label{eq:Sug},
\end{align}
{\color{black}where $\tilde{h}$ is the function defined in \eqref{def_htilde}, {and} \begin{equation*}
	K_1 :=\sup_{\e>0} \Big\|\frac{1}{2}\sigma^2|\psi_\e+\psi_0|+|\tilde{h}(\psi_{\e},\psi_0)|\Big\|_{L^{\infty}(0,T)};
\end{equation*}
{since} $\psi_{0}$ is locally bounded on $\mathbb{R}_+$ by Lemma \ref{lem_aggiunto}, $K_1$ is finite by \eqref{unif_bound} and   $\tilde{h}$ {is continuous}, see Appendix \ref{section:AlessE+U} (in particular, the proof of Lemma \ref{lemA3}).}
 Then
\begin{align*}
&\bigg|\int_0^T \Big(G(s,\psi_\e(s))\xi^\e_0(T-s)-G(s,\psi_0(s))\xi^0_0(T-s)\Big) ds\bigg|\\&\qquad \le
\int_0^T |G(s,\psi_\e(s))-G(s,\psi_0(s))|\,\xi^0_0(T-s) ds  +  \Big(\sup_{\e>0}\|G(\cdot,\psi_\e)\|_{L^\infty(0,T)}\Big) \int_0^T \,|\xi_0^\e(T-s)-\xi^0_0(T-s)| ds,
\end{align*}
where the last two terms tend to zero as $\e \to 0$ by \eqref{eq:Sug} and Assumption \ref{Assumption:2.1}.
%Let $\Lambda(u):=-\lm u +\half \sigma^2 u +V_1(u)$.
% then the equation $f(t)+\Lm(u)=0$ has exactly two roots since $V_1$ (and hence $\Lm$) is convex, and $\Lm(0)=0$.
%and the zero function is a subsolution to \eqref{eq:Vint} (because $f\le 0$),
%hence $\psi_{\e}(t)\le 0$ by the comparison principle. 
%Hence we also have convergence for $G(s,\psi_{\e}(s))$ Lebesgue a.e.\,(since we have convergence at the continuity points of $f$
%away from zero)
%in \eqref{eq:mgf}, so (from the bounded convergence theorem),
%we have 
Hence {\color{black}from \eqref{lim_eq}} and \eqref{eq:mgf}  we see that
\begin{align}\nn
\lim_{\e \to 0}\Ex\Big[e^{\int_0^T f(T-s) V^{\e}_s ds}\Big] &=\lim_{\e \to 0}e^{\int_0^T  G(s,\psi_{\e}(s)) \xi^{\e}_0(T-s) ds} \\&
= e^{\int_0^T  G(s,\psi_{0}(s)) \xi^0_0(T-s) ds}~  e^{\lm \int_0^T  \psi_{0}(s) \xi^0_0(T-s) ds}\label{eq:this}.
\end{align}
%since $\psi_{\e}$ converges to $\psi_{0}$ Lebesgue a.e.\,on $(0,T]$ (and $\xi^{\e}_0$ converges pointwise to $\xi^{0}_0$ by assumption).  

We now characterize the process
which has \eqref{eq:this} as its Laplace transform.
\\
For a  L\'{e}vy process $Z=(Z_t)_{t\ge0}$ with L\'{e}vy triple $(-\lm,\sigma^2,\nu)$, 
$e^{u Z_t- \Lambda(u) t }$ is an $\mc{F}^Z_t-$martingale for $u \le 0$, where 
\begin{equation}\Lm(u)=-\lm u +\half \sigma^2 u^2 +V_1(u);\label{agg_ref}
\end{equation}
this is a consequence of the stationary and independent increments property, together with Theorem 25.17 in \cite{Sato99} (see \eqref{precise} below for an analogous argument).
%we know that $\Ex(e^{u Z_t})<\infty$ for $u < M$, and (for $u$ in this range)
From Proposition \ref{Prop_LT} in Appendix \ref{section:LT} applied to the spectrally negative process $-Z$ (see also  e.g.\,Theorem 46.3 in \cite{Sato99} or Eq. (2.5) in \cite{KKR13} for an alternative proof), we know that
\bq
\Ex[e^{-q H_a}]~ e^{-\Lambda^{-1}(q)a}~ e^{\Lambda^{-1}(q)|a|}\label{eq:OST}
\eq
for $q\ge 0$ and $a\le0$, where $H_a:=\inf \lb t \ge 0: Z_t<a\rb$ and $\Lambda^{-1}(q)\le 0$ denotes the inverse function of $\Lambda$. %smallest root of $\Lm(u)=q$
 %(since we have a \textit{lower} barrier here)
%, and 
% by standard theory on hitting times for L\'{e}vy processes (since $\Lm'(0)=-\lm+V_1'(0)=-\lm<0$) $H_{a}<\infty$ a.s. for $a\le0$, {\color{red}see also Remark 46.1 in \cite{Sato99}.}
% {\\\color{cyan}Ale: Given the detailed Appendix E and the precise reference to Sato, the next sentence does not add anything, so I suggest to remove it to improve the flow\\
 %We can use the optional stopping theorem to prove \eqref{eq:OST}, see e.g. Eq 2.5 in \cite{KKR13} for details, and the limiting argument used to prove this using the dominated convergence theorem \footnote{in \cite{KKR13} they take the largest root since they are considering an upper barrier).
%and \eqref{eq:KK} also holds for $q<0$ if $q$ is greater than the minimum value attained by $\Lm(.)$.  

\sk
Now let $X_t=H_{-t}$, which is a L\'{e}vy subordinator, see Remark \ref{rem_subordinator} and the references therein.  Then from the i.i.d. property for L\'{e}vy processes, \eqref{eq:OST} and the fundamental theorem of calculus,  considering a right-continuous non-positive piecewise constant function ${f}$ we have, for any  continuously differentiable non-decreasing function $g$ starting from 0,
\begin{align}\label{final_recall}
\notag\Ex\Big[e^{\int_0^T {f}(T-s)d(X_{g(s)})} \Big]
&~e^{\int_0^T \Lambda^{-1}(|{f}(T-s)|)g'(s)ds} ~ e^{\int_0^T \Lambda^{-1}(-f(s))g'(T-s)ds} \\&
~
e^{\int_0^T \psi_0(s)g'(T-s)ds},
%~ e^{\int_0^T\psi_0(s) g'(T-s)ds}
\end{align} 
% {\\\color{cyan}Ale:  $f$ should be right-continuous because then the map $s\mapsto f(T-s)$ inside the integral is left-continuous. Also, (17) is immediate using i.i.d, (16) and FTC, so we could also remove the references to [CT04] and [AAR25] here. \vspace{2mm}\\
% }
% But $\Phi(-f(.))$ is 
where in the last step we use that  $\Lambda^{-1}(-f(\cdot))$ is $\psi_0(\cdot)$ by Eqs. \eqref{lim_eq} and \eqref{agg_ref} (see also Lemma 15.1 in \cite{CT04}, which is used in \cite{AAR25}).
\vspace{2mm}\\
At this  point, for $0\le s_1<\cdots<s_n<T$, choosing a right-continuous map $f$ such that $$f(T-s)=(u_1+u_2+...+u_n) 1_{\{0 \le s \le s_1\}}+  (u_2+...+u_n) 1_{\{s_1 < s \le s_2\}}+...+ u_n 1_{\{s_{n-1} < s \le s_{n}\}},\quad s\in[0,T]$$
%\[
	%	f(s)=(u_1+u_2+...+u_n) 1_{\{T-s_1 \le s\}}+  (u_2+...+u_n) 1_{\{T-s_2 \le  s <  T-s_1\}}+...+ u_n 1_{\{t_{n-1} < s \le s_{n-1}\}}
%\]
with $u_1,u_2,\dots,u_n\le 0$, we see that
\bq
\int_0^T f(T-s) V^{\e}_s ds ~u_1 A^{\e}_{s_1}+\,...\,+ u_n A^{\e}_{s_n},\quad
\int_0^T f(T-s) dX_{g(s)} ~u_1 X_{g(s_1)}+\,...\,+ u_n X_{g(s_n)},\nn
\eq
where $A^{\e}_t=\int_0^t V^{\e}_s ds$. Therefore
%\sum_{i=1}^n i u_i 1_{s_i < s \le s_{i+1}}$, 
 by \eqref{eq:this} and \eqref{final_recall} (again by Problem 30.4 in \cite{Bill86})
 we deduce  that the finite-dimensional distributions
of $A^{\e}_{\cdot}=\int_0^{(\cdot)} V^{\e}_s ds$ converge weakly to those of the time-changed L\'{e}vy process $X_{g(t)}$
with $g'(t)=\lm \xi^0_0(t)$ and $g(0)=0$. The proof is now complete.
 \end{proof}
%VIX, VIX inversion, MLE, QRHeston

\subsection{Weak functional convergence in the $M_1$ topology}\label{sub_func}
The purpose of this subsection is to strengthen the convergence result of the finite-dimensional distributions established in Theorem \ref{thm_main}  to a functional setting. More precisely, recalling  the process $X_{g(\cdot)}$ and the integrated variances $A^\e_\cdot=\int_0^{(\cdot)}V^\e_s ds$  in the statement of Theorem \ref{thm_main},  our aim is to prove the weak convergence in law of  $A_\cdot^\e$  to $X_{g(\cdot)}$ as $\e\to 0$ in the path-space of c\`adl\`ag functions $D(\mathbb{R}_+;\mathbb{R})$, endowed with the $M_1$ topology. 
Notice that we work with trajectories defined on the whole  non-negative half-line, instead of fixing a finite time horizon as in \cite{AAR25} or \cite{AC24}, where the $M_1$ topology is also used. \\We consider $D(\mathbb{R}_+;\mathbb{R})$ instead of the space of continuous functions $C(\mathbb{R}_+;\mathbb{R})$ -- where $(A_\cdot^\e)_\e$ take values -- because $X_{g(\cdot)}$ is a c\`adl\`ag process that, by Theorems 21.1 and 21.3 in  \cite{Sato99}, exhibits jumps in an event with positive probability (unless in the trivial case $\xi_0^0(t)=0$ for a.e. $t\ge0$). 
The reason we focus on the $M_1$ topology rather than the arguably more common $J_1$ topology is that, as a consequence of, e.g., Proposition 2.1 a) in Chapter \upperRomannumeral{6} of \cite{JS}, $C(\mathbb{R}_+;\mathbb{R})$ is a closed subset of $D(\mathbb{R}_+;\mathbb{R})$ in the $J_1$ topology. By the Portmanteau theorem, any weak limit probability of $(A_\cdot^\e)_\e$ is then supported on $C(\mathbb{R}_+;\mathbb{R})$, which again excludes the time-shifted subordinator $X_{g(\cdot)}$.

Before stating our main result in Theorem \ref{cor_functional}, we recall some important definitions and properties of $M_1$. A comprehensive analysis of this and other topologies on spaces of c\`adl\`ag functions can be found in \cite{W}, see in particular Chapter 12. We start by presenting the definition of the topology $M_1^{(T)}$ on the space $D([0,T];\mathbb{R})$ of real-valued c\`adl\`ag functions defined on $[0,T]$, for every $T>0$.
\begin{defn}
    Fix $T>0$. Given $x\in D([0,T];\mathbb{R})$, the \emph{thin graph} of $x$ is defined by
    \[
    \Gamma_x=\{(z,t)\in\mathbb{R}\times [0,T] : z\in [x(t-),x(t)]\},\quad \text{where $x(0-)=x(0)$}.
    \]
    The couple $(\Gamma_x, \le)$ is a totally ordered set, where $\le$ is the relation given by:\[
(z_1,t_1)\le (z_2,t_2) \qquad \Longleftrightarrow \qquad \Big[t_1<t_2\Big] \quad \text{or} \quad \Big[t_1=t_2 \,\,\text{ and }\,\,
|x(t_1-)-z_1|\le |x(t_1-)-z_2|\Big].
    \]
    A \emph{parametric representation} of $x$ is a continuous nondecreasing function $(u,r)$ mapping $[0,1]$ onto $(\Gamma_x,\le)$; the set of such parametric representations is denoted by $\Pi(x)$.\vspace{1mm}\\
    For every $x,y\in D([0,T];\mathbb{R})$, define the metric $d_T$ as follows:
    \begin{equation}\label{def_measure}
        d_T(x,y)=\inf_{(u_1,r_1)\in \Pi(x), (u_2,r_2)\in \Pi(y)} \max\{\|u_1-u_2\|_1, \|r_1-r_2\|_1\},
    \end{equation}
    where $\|\cdot\|_1$ denotes the uniform norm in $[0,1]$. 
    The $M_1^{(T)}$ topology is the topology generated by $d_{T}$.
\end{defn}
\noindent
We refer the reader to Theorem 12.3.1 in \cite{W} for the proof that $d_T$ is indeed a metric on $D([0,T];\mathbb{R})$. 

In the sequel, for every $T>0$, we denote by $r_T\colon D(\mathbb{R}_+;\mathbb{R}) \to  D([0,T];\mathbb{R}) $ the restriction operator given by $[r_T(x)] (t)= x(t)$ for $t\in[0,T]$. This operator enables us to introduce the $M_1$ topology on $D(\mathbb{R}_+;\mathbb{R})$, hence for functions defined on the non-negative half-line; see also Section 12.9 in \cite{W}.
\begin{defn}
The $M_1$ topology on $D(\mathbb{R}_+;\mathbb{R})$ is the topology generated by the metric   
    \begin{equation}\label{def_distance}
d_{M_1}(x, y) = \int_{0}^{\infty} e^{-T}\min \big\{d_T(r_T(x), r_T(y)) , 1\big\}\,dT, 
\end{equation}
where, for every $T>0$, $d_T$ is defined in \eqref{def_measure}.

\end{defn}
Despite the metric $d_T$ in \eqref{def_measure} being incomplete, the space $(D([0,T];\mathbb{R}),M_1^{(T)})$ is Polish, see Section 12.8 in \cite{W}. It follows that  also $(D(\mathbb{R}_+;\mathbb{R}), M_1)$ is Polish, and we refer to Theorem 2.6 in \cite{Wpaper} for an analogous argument for the $J_1$ topology. 
As a consequence of the definition of $d_{M_1}$ in \eqref{def_distance} and  Theorem 12.3.2 in \cite{W}, the $M_1$ topology is coarser than the $J_1$ topology. Since they both provide Polish (hence Lusin) spaces, by Theorem 11.5.3 in \cite{W} they generate the same $\sigma-$algebra, which coincides with the classical $\sigma-$algebra generated by the evaluation maps on $D(\mathbb{R}_+;\mathbb{R})$, see also Theorem 1.14 in Chapter \upperRomannumeral{6} of \cite{JS}. \vspace{2mm}

Denote by $\mathbb{P}^\e$ and $\mathbb{P}$ the probability measures induced by $A_\cdot^\e$ and $X_{g(\cdot)}$ on $\sigma(M_1)$, respectively.
The next lemma enables us to reduce the study of the weak convergence of   $\mathbb{P}^\e\, \Longrightarrow \,\mathbb{P}$ on $(D(\mathbb{R}_+;\mathbb{R}), M_1)$ from an infinite to an arbitrary finite time horizon. More precisely, we have the following.
\begin{lem}\label{lem_reduction}
    For any $T>0$, let $\mathbb{P}r_T^{-1}$ be the pushforward probability measure on $\sigma(M_1^{(T)})$ of $\mathbb{P}$ by the restriction operator $r_T$, with an analogous notation for $\mathbb{P}^{\e}$. Then $\mathbb{P}^\e\,\Longrightarrow \,\mathbb{P}$ on $(D(\mathbb{R}_+;\mathbb{R}), M_1)$ if and only if, for every $T>0$,  $\mathbb{P}^\e r_T^{-1}\,\Longrightarrow \,\mathbb{P}r_T^{-1}$ on $(D([0,T];\mathbb{R}), M_1^{(T)})$.
\end{lem}
\begin{proof}
  Since $X_{g(\cdot)}$ is a time-changed L\'evy process with continuous time change $g$ (we recall from the statement of Theorem \ref{thm_main} that $g(\cdot)=\lm \int_0^{(\cdot)} \xi^0_0(u)du$), for any fixed $T>0$,
\[
\mathbb{P}\big(\{x\in D(\mathbb{R}_+;\mathbb{R}) : x(T)-x(T-)\neq 0\}\big)=0.
\]
Moreover, by Theorem 12.9.3 in \cite{W}, the restriction map $r_T$ is continuous at all $x\in D(\mathbb{R}_+;\mathbb{R})$ that do not exhibit a jump at $T$. Combining these two facts, the proof can be completed using the same arguments (based on the continuous mapping and Portmanteau theorems) as in Theorem 2.8 of \cite{Wpaper}, which treats the case of the $J_1$ topology.
\end{proof}
At this point, we are ready to state the main result of this subsection. Its proof relies on tightness characterizations of sets of probability measures on $(D([0,T];\mathbb{R}), M_1^{(T)})$ for $T>0$, established in Theorem 12.12.3 in \cite{W}. These conditions are particularly simple to apply since the processes $(A_\cdot^\e)_{\e>0}$ are non-decreasing.
\begin{thm}\label{cor_functional}
The integrated variance processes $A_\cdot^\e$ converge weakly to the time-shifted L\'evy process $X_{g(\cdot)}$ defined in Theorem \ref{thm_main} on $(D(\mathbb{R}_+;\mathbb{R}), M_1)$,  that is, $\mathbb{P}^\e\,\Longrightarrow \,\mathbb{P}$ on $(D(\mathbb{R}_+;\mathbb{R}), M_1)$ as $\e\to 0$.
\end{thm}
\begin{proof}
According to Lemma \ref{lem_reduction}, it is sufficient to prove that 
\begin{equation}\label{con_fth}
    \mathbb{P}^\e r_T^{-1}\,\Longrightarrow \,\mathbb{P}r_T^{-1} \text{ on $(D([0,T];\mathbb{R}), M_1^{(T)})$, for every $T>0$.}
\end{equation}
In turn, by Theorem \ref{thm_main} and Theorem 11.6.6 in \cite{W},  \eqref{con_fth} is satisfied if the family of probability measures $(\mathbb{P}^\e r_T^{-1})_\e$ is tight on $(D([0,T];\mathbb{R}), M_1^{(T)})$. According to Theorem 12.12.3 in \cite{W}, sufficient (and necessary) conditions for tightness in the space $(D([0,T];\mathbb{R}), M_1^{(T)})$ are given by
\[
\lim_{R\to\infty}\sup_{\e>0}\mathbb{P}\Big(\sup_{t\in[0,T]}|A^\e_t|>R\Big)=0
\]
and 
\[
    \lim_{\delta\to 0 }\sup_{\e>0}\mathbb{P}(w_T(A_\cdot^\e, \delta)>\eta)=0,\quad \text{for any fixed $\eta>0$}.
\]
Here, denoting by $\text{dist}(t_2,[t_1,t_3])=\inf_{t\in[t_1,t_3]}|t_2-t|$ the distance between $t_2$ and the segment joining $t_1$ and $t_3$ (with $t_i\in\mathbb{R},\,i=1,2,3$),  we set, for any $\delta>0$,
\begin{align*}
    &w_T(A_\cdot^\e,\delta) =  \max\Big\{\sup_{0\le t \le T}w_T'(A^\e_\cdot,t,\delta) , v_T(A^\e_\cdot,0,\delta), v_T(A_\cdot^\e,T,\delta)\Big\};\\&
w'_T(A_\cdot^\e,t,\delta) = \sup_{\max\{0,t-\delta\} \le t_1 < t_2 < t_3 \le \min\{t+\delta, T\}} 
\text{dist}\big(A^\e_{t_2}, [A^\e_{t_1},A^\e_{t_3}]\big),\quad t\in [0,T];\\&
v_T(A^\e_\cdot,t,\delta) = \sup_{\max\{0, t-\delta\} \le t_1 \le t_2 \le \min\{t+\delta, T\}} 
|A^\e_{t_1} - A^\e_{t_2}|,\quad t\in [0,T].
\end{align*}
Since the processes $A_\cdot^\varepsilon$ start from $0$ and are non-decreasing, being obtained by integrating the non-negative variance processes $V^\varepsilon_\cdot$, as already observed in Section 4.3 of \cite{AAR25}, these conditions simplify. In particular, the former reduces to
\begin{equation}\label{cond_1}
\lim_{R \to \infty} \sup_{\varepsilon>0} \mathbb{P}\big( A^\varepsilon_T > R \big) = 0.
\end{equation}
As for the latter,  considering that, for every $t\in[0,T]$,  $v_T(A^\e_\cdot,t,\delta)=A^\e_{\min\{t+\delta,T\}}-A^\e_{\max\{t-\delta,0\}}$ and 
$$
A^\varepsilon_{t_2}\in[A^\varepsilon_{t_1},A^\varepsilon_{t_3}] \text{ for all $t_1\le t_2\le t_3$, \quad hence $w'_T(A^\varepsilon_\cdot,t,\delta)=0$},$$  it reads
\begin{equation}\label{cond_2}
\lim_{\delta \to 0} \sup_{\varepsilon>0} \mathbb{P}\Big(\max\big\{ A^\varepsilon_\delta, A^\varepsilon_T - A^\varepsilon_{T-\delta}\big\} > \eta \Big) = 0,
\quad \text{for all fixed $\eta>0$}.
\end{equation}
To complete the proof, it will then be sufficient to verify \eqref{cond_1} and \eqref{cond_2}. For every $\e>0$ and $T\ge0$, integrating \eqref{SVE_jumps} on $[0,T]$, an application of the stochastic Fubini theorem yields (see Lemma 3.2 in \cite{ACLP}, and also Lemma 2.1 in \cite{A21})
\begin{equation*}
    A^\e_T= \int_0^T \xi_0^\e (s)ds+\int_0^T\kappa_\e(T-s) \bigg(
    \sigma\int_0^s \sqrt{V_r^\e}d W_r + \tilde{J}^\e_s 
    \bigg)
    ds,
\end{equation*}
where we recall that $\tilde{J}^\e=(\tilde{J}^\e_t)_{t\ge0}$ is the process defined in \eqref{tildeJ_mg}.
Taking expectations, 
\begin{equation}\label{exp_functional}
    \mathbb{E}[A^\e_T]=\int_0^T\xi^\e_0 (s)ds+\int_0^T\kappa_\e(T-s) \mathbb{E}\bigg[
    \sigma\int_0^s \sqrt{V_r^\e}d W_r + \tilde{J}^\e_s 
    \bigg]
    ds=\int_0^T\xi^\e_0(s) ds,
\end{equation}
where the last step follows from the martingale property of the processes  $\int_0^\cdot \sqrt{V_r^\e}dW_r$ and $\tilde{J}^\e$, ensured by \eqref{L1integrability} and \eqref{tildeJ_mg}, respectively. Since by Assumption \ref{ass1} the maps $\xi_0^\e(\cdot)$ are uniformly bounded,  by Markov's inequality, \eqref{exp_functional} implies \eqref{cond_1}. As for \eqref{cond_2}, it follows by the same arguments, once we notice that, by \eqref{exp_functional},  
\[
\sup_{\e>0}\max \Big\{\mathbb{E}[A^\e_\delta],\mathbb{E}[A^\e_T-A^\e_{T-\delta}]\Big\}\le  \delta \sup_{\e>0,t\in [0,T]}|\xi_0^\e(t)|\longrightarrow 0 \quad \text{as }\delta\to 0.
\]
Therefore, the proof is complete.
\end{proof}
\begin{rem}\label{cor_continuous}
   If we consider the continuous solution $V_\cdot^\e$ to \eqref{eq:V_alt} without jumps, i.e., $\nu(dx)=0$ (hence $\tilde{J}^\e=0$), then $V^\e_\cdot$ coincides with \eqref{eq:Org} in Theorem \ref{cont_prop} of Section \ref{asy_marginal}. Recalling the arguments in Remark \ref{rem:2.1}, an application of Theorem \ref{thm_main} yields that the finite-dimensional distributions of 
$
A^\e_\cdot = \int_0^{(\cdot)} V^\e_s\, ds
$
converge to those of $X_{\lambda \theta (\cdot)}$, where $X_\cdot$ is the L\'evy subordinator given by
\[
X_t = \inf\Big\{ s \ge 0 : B_s + \frac{\lambda}{\sigma} s = \frac{1}{\sigma} t \Big\},
\]
for a standard one-dimensional Brownian motion $B$. 
By \eqref{constr_IG}, we infer that $X_{\lambda \theta (\cdot)}$ is an $\mathrm{IG}$ process with parameters $(\theta, \sigma^{-2}\lambda^2 \theta^2)$. Therefore, in the case of continuous dynamics, Theorem \ref{thm_main} extends Theorem \ref{cont_prop} by proving the convergence of the finite-dimensional distributions, rather than only that of the marginals at a fixed time $t$. Furthermore, Theorem \ref{cor_functional} implies that the laws of $A^\e_\cdot$ converge, as $\e \to 0$, to the law of $X_{\lambda \theta (\cdot)}$ weakly on the path space $(D(\mathbb{R}_+;\mathbb{R}), M_1)$.
\end{rem}

 \begin{appendix}

\section{Proof of Lemma \ref{lem:E+U}}
\label{section:AlessE+U}
\renewcommand{\theequation}{A-\arabic{equation}}
\setcounter{equation}{0}

%	For every $\e>0$, consider the Riccati-Volterra equation
%	{\color{red} Consider a function $f\in L_{\text{loc}}^\infty (\mathbb{R}_+;\mathbb{R}_-)$, which is therefore not necessarily   left-continuous or piecewise constant, unlike in the proof of Theorem \ref{thm_main}.} 
We first recall the Riccati--Volterra integral equation in \eqref{eq:G}-\eqref{eq:Vint}, which we re-write as
	\begin{align}\label{RV_eq}
	\notag	\psi_{\e}(t)&=\int_{0}^{t}\kappa_{\e}(t-s)f(s) ds + \int_{0}^{t}\kappa_{\e}(t-s)\bigg(\frac{1}{2}\sigma^2\psi_{\e}^2(s) + \int_{\mathbb{R}_+}(e^{x\psi_{\e}(s)}-1- x\psi_{\e}(s))\nu(dx)\bigg)ds\\
		& =
		(\kappa_{\e}\ast f) (t) + \Big(\kappa_\e \ast \Big(\frac{1}{2}\sigma^2\psi^2_{\e}+V_1(\psi_{\e})\Big)\Big)(t),\quad t\ge 0. 
	\end{align}
%{\color{cyan}Ale: I removed the map $\bar{G}$, as it was unnecessary and we already have a lot of these $G$ maps ($\widetilde{G},\, G$)}\\
% and $\lambda\kappa_{\e}$ is the resolvent of the second kind of $\frac{\lambda}{\e}t^{\alpha-1}$, for $\lambda>0$. 
By Theorem 3.1, Chapter 5 in \cite{GLS90}, $\kappa_{\e}$ is completely monotone,
and by Theorem 2.2, Chapter 2 in \cite{GLS90}, %Note that the map
 $t\mapsto (\kappa_{\e}\ast f) (t) $ is continuous on $\mathbb{R}_+$ (as $f\in L^\infty_{\text{loc}}(\mathbb{R}_+;\mathbb{R}_-)$).  Moreover, the function $\widetilde{G}$ defined by the relation
\[
	\widetilde{G}(w)-\frac{1}{2}\sigma^2w^2 = \begin{cases}
		0,& w>0\\
		\int_{\mathbb{R}_+}(e^{xw}-1- xw)\nu(dx), & w\le 0
	\end{cases}
\]
is continuous and non-negative on $\mathbb{R}$. Then by Theorem 1.1, Chapter 12 in \cite{GLS90}, there exists a continuous noncontinuable %\footnote{see Page 343 in \cite{GLS90} for definition.} 
 local solution $\widetilde{\psi}_{\e}$ of the equation
\begin{equation}\label{RV_mod}
	\widetilde{\psi}_{\e}(t)=(\kappa_{\e}\ast f) (t) + (\kappa_\e \ast \widetilde{G}(\widetilde{\psi}_{\e}))(t),\quad t \in [0,T_\text{max}),
\end{equation}
for {some} $T_\text{max}>0$. 
%Additionally, 
%\begin{equation}\label{explosive}
%	\limsup_{t\to T_\text{max}}|\widetilde{\psi}_{\e}(t)|=\infty \quad \text{ if }\quad T_{\text{max}}<\infty.
%\end{equation}
Then, {\color{black}noting that $\widetilde{G}=\frac{1}{2}\sigma^2(\cdot)^2+V_1(\cdot)$ on $\mathbb{R}_-$}, from the following lemma    we know that $\widetilde{\psi}_{\e}$ also solves \eqref{RV_eq} on $[0,T_\text{max})$:

%that \bl{$\widetilde{\psi}_{\e}$} is non-positive (TO DO). 

\begin{lem}\label{lem2.4}
%Any solution of \eqref{eq:Vint} 
$\widetilde{\psi}_\e$ is non-positive.
%, $\psi_{\e}(t)\le 0$ for all $t\ge 0$.
% where $\Phi_0^-(t)$ is the negative root of $G(t,.)$
%(defined below Eq \eqref{eq:Vint}). 
\end{lem}
%{\color{cyan}Ale: I removed the old proof which used $\Phi^0_{\pm}(s)$ because of the critical issue we discussed via email.\\}
\begin{proof}
{For convenience we define} $h\colon  \mathbb{R}\to \mathbb{R}_-$ by 
\begin{equation}\label{def_h}
	h(w)=\begin{cases}
		\frac{1}{w}\int_{\mathbb{R}_+} (e^{wx}-1-wx)\nu(dx),&w < 0,\\
		0,&w\ge0;
	\end{cases}
\end{equation}
(which is continuous and non-positive), so
$\frac{1}{2}\sigma^2 w^2+w\cdot h(w)=\widetilde{G}(w)$ for {$w\in \mathbb{R}$}.

Then, for every $T\in (0,T_\text{max})$, from \eqref{RV_mod} 
\[
\widetilde{\psi}_\epsilon(t)=
(\kappa_{\e}\ast f) (t) + \int_{0}^{t}\kappa_\e(t-s) \bigg(\frac{1}{2}\sigma^2 \widetilde{\psi}_{\e}(s)+ h(\widetilde{\psi}_{\e}(s))\bigg)\widetilde{\psi}_{\e}(s)ds
,\quad t\in [0,T].
\]
By Remark B.6 in \cite{AE19} (which allows to consider a possibly discontinuous function $f$) and recalling that $f\le 0$, this reformulation enables us to use Theorem C.1 in \cite{AE19} and conclude that $\widetilde{\psi}_\e\le 0$ in $[0,T_\text{max})$, as $T$ is arbitrary.
\end{proof}

\sk
\sk
%\bl{To ease notation going forward}, we  
Thanks to the discussion above and Lemma \ref{lem2.4}, we can consider a continuous noncontinuable non-positive local solution $\psi_\e$ to \eqref{RV_eq}. In the next lemma, we show that $\psi_\e$ is in fact a global solution, i.e., it is defined on the entire $\mathbb{R}_+$.
\begin{lem}
    Any continuous noncontinuable $\mathbb{R}_--$valued solution $\psi_\e$ to \eqref{RV_eq} is globally  defined on $\mathbb{R}_+$.
\end{lem}
\begin{proof}
   Recalling the definition of $V_1$ in \eqref{eq:G}, we notice  that {\color{black}$\frac{1}{2}\sigma^2\psi_\e^2+V_1(\psi_\e)\ge0$} in $[0,T_\text{max})$, so \eqref{RV_eq} yields
\begin{equation}\label{bound}
	(\kappa_{\e}\ast f) (t)\le \psi_{\e}(t) \le 0,\quad t\in [0,T_\text{max}). 
\end{equation}
{\color{black}
Since $\kappa_{\e}\ast f$ is continuous and dominates $\psi_{\e}$ on $[0,T_{\text{max}})$,
$\psi_{\e}$ cannot explode to $-\infty$ at $T_{\text{max}}$. Then, given that by Theorem 1.1, Chapter 12 in \cite{GLS90},
\[
	\limsup_{t\to T_{\text{max}}}|\psi_\e(t)|=\infty\quad \text{if }T_{\text{max}}<\infty,	
\]
we conclude that $T_\text{max}=\infty$.}
%{\color{cyan}\\Ale : I believe it is important to include this $\limsup$, otherwise the boundedness of $\psi_e$ does not imply $T_{max}=\infty$}
% which combined with \eqref{explosive} gives  $T_\text{max}=\infty$.
\end{proof}
To complete the proof of Lemma \ref{lem:E+U}, it only remains to establish uniqueness for \eqref{RV_eq}. We do this in the following lemma.
\begin{lem}\label{lemA3}
    If $\psi_1$ and $\psi_2$ are two continuous $\mathbb{R}_--$valued global solutions of \eqref{RV_eq},  then $\psi_1=\psi_2$ on $\mathbb{R}_+$
\end{lem}
\begin{proof}
    Consider $\psi_1$ and $\psi_2$ as in the statement of the lemma. Then {setting} $\delta=\psi_1-\psi_2$, \begin{align}\label{eq_uniq}
\nn	\delta(t)=&\int_{0}^{t}\kappa_{\e}(t-s)\bigg(\frac{1}{2}\sigma^2(\psi_1(s)+\psi_2(s))\delta(s) \\&+ \int_{\mathbb{R}_+}\Big(e^{x\psi_1(s)}-e^{x\psi_2(s)}-x(\psi_1(s)-\psi_2(s))\Big)\nu(dx)\bigg)ds.
\end{align}
Now let $\mathbb{R}^2_-=\{(w_1,w_2)\in \mathbb{R}^2, w_1\le0 \text{ and }w_2\le 0\}$ denote the negative quadrant of the plane, and define the auxiliary function $\tilde{h}\colon \mathbb{R}^2_-\to \mathbb{R}$ by 
\begin{equation}\label{def_htilde}
\tilde{h}(w_1,w_2)=\begin{cases}
	\frac{1}{w_1 -w_2 }\int_{\mathbb{R}_+} (e^{xw_1}-e^{xw_2}-x(w_1 -w_2))\nu(dx),&w_1\neq w_2,\\
	\int_{\mathbb{R}_+}x(e^{xw_1}-1)\nu(dx),&\text{otherwise}.
\end{cases}
\end{equation}
%{\color{cyan}Ale : as I told you via email, I switched back to my old continuity proof, as your  proof relies on $\int_{\mathbb{R}_+} (e^{wx}-wx)\nu(dx)$, which is not finite in general (this is a secondary  issue)\\}
The map $\tilde{h}$ is non-positive on its domain $\mathbb{R}^2_-$, because $w\mapsto e^{xw}-xw$ is nonincreasing  on $\mathbb{R}_-$ for every $x\in \mathbb{R_+}$. 
{\color{black}By the dominated convergence theorem, $\tilde{h}$ is continuous on 
$\mathbb{R}^2_- \setminus \{(w,w) : w \le 0\}$. To prove the continuity at the points $(w_1,w_2)\in\mathbb{R}^2_-$ with $w_1=w_2$, consider two sequences $(w_{1,n})_n, (w_{2,n})_n\subset \mathbb{R}_-$ such that \[
	w_{1,n}\neq w_{2,n}\quad  \text{and}\quad  \lim_{n\to \infty }w_{1,n}=w_1=w_2=\lim_{n\to \infty }w_{2,n}.
	\]
	Without loss of generality, suppose that $w_{1,n}>w_{2,n}$ for every $n\in \mathbb{N}$. We then compute, using the inequality $|e^u-1-u|\le |u|^2$ for $u\in \mathbb{R}_-$,
	\begin{align*}
		&	\Big|\tilde{h}(w_{1,n},w_{2,n})-\tilde{h}(w_1,w_2)\Big|\\
		&\quad \le
		\frac{1}{|w_{1,n}-w_{2,n}|}\int_{\mathbb{R}_+}
		e^{xw_{1,n}} \Big|
		e^{x(w_{2,n}-w_{1,n})}-1- x(w_{2,n}-w_{1,n})
		\Big|
		\nu(dx)+
		\int_{\mathbb{R}_+}x\Big|e^{xw_{1,n}}-e^{xw_{1}}\Big|\nu(dx)
		\\
		&\quad \le 
		\bigg(	\int_{\mathbb{R}_+}|x|^2\nu(dx)|\bigg)|w_{2,n}-w_{1,n}|+\text{o}(1)
		\underset{n\to\infty}{\longrightarrow}0.
	\end{align*}
	Considering that $\tilde{h}(w_{1,n},w_{1,n})\to \tilde{h}(w_1,w_2)$ by dominated convergence, the previous computations prove the continuity of $\tilde{h}$ at $(w_1,w_2)$ with $w_1=w_2$. }
%The map $\tilde{h}$ is non-positive in its domain $\mathbb{R}^2_-$, because $F(w):=e^{xw}-xw$ is nonincreasing  in $\mathbb{R}_-$ for every $x\in \mathbb{R_+}$.  Moreover, $\frac{F(w+h)-2F(w)+F(w-h)}{h^2}$ is monotonically decreasing in $h$, hence (from the decreasing version of the monotone convergence theorem),
%$V_1(w)=$ is twice differentiable in $w$ for $w<0$,
%so (from the Taylor remainder theorem) $\tilde{h}(w_1,w_2)=V_1'(w_*)$ for some $w^*\in [w_1,w_2]$
%and (since $V'_1$ is continuous), $\tilde{h}$ is continuous.

\sk

From the definition of $\tilde{h}$, we have
\[
\int_{\mathbb{R}_+}\Big(e^{x\psi_1(s)}-e^{x\psi_2(s)}-x(\psi_1(s)-\psi_2(s))\Big)\nu(dx)
	=
	\tilde{h}(\psi_1(s),\psi_2(s))\delta(s),\quad s\ge 0,
\]
hence (from \eqref{eq_uniq}) $\delta$ solves the linear VIE:
\[
		\delta(t)=\int_{0}^{t}\kappa_{\e}(t-s)\bigg(\frac{1}{2}\sigma^2(\psi_1(s)+\psi_2(s))
		+	\tilde{h}(\psi_1(s),\psi_2(s))\bigg)\delta(s)ds.
\] 
This equation admits  $\delta\equiv 0$ as its unique solution by {the first part of} Theorem C.1 in \cite{AE19}, whence we conclude that $\psi_1=\psi_2$. The proof is now complete.
\end{proof}
 %This proves that \eqref{RV_eq} has a unique continuous global $\mathbb{R}_--$valued solution.

\section{Proof of Lemma \ref{lem:Ale}}
\label{section:Conv}
% without $\lm$ restriction}
%as $\epsilon\to 0$}
\renewcommand{\theequation}{B-\arabic{equation}}
\setcounter{equation}{0}
We recall that $f$ is a locally bounded non-positive function defined on $\mathbb{R}_+$, which we also write as $f\in L_{\text{loc}}^\infty (\mathbb{R}_+;\mathbb{R}_-)$. Moreover, for every $\e>0$, $\psi_\e$ is the unique solution of the deterministic VIE \eqref{eq:G}-\eqref{eq:Vint}, whose well-posedness is established in Lemma \ref{lem:E+U}.
\subsection{Relative compactness in $L^1$}
Fix $T>0$ and a sequence $(\e_n)_n\subset (0,\infty)$ which converges to 0. 
\begin{lem}\label{lem_RFK}
 $(\psi_{\e_n})_n$ admits a convergent subsequence in $L^1(0,T)$. 
\end{lem}
\begin{proof}
%Observe that 
\bl{$(\psi_{\e_n})_n$ is bounded in $L^1(0,T)$ because (\bl{from} \eqref{bound}) we know that}
%and \cite[Theorem 2.2, Chapter 2]{Gr}
%\bl{with $p=\infty$}
\begin{equation}\label{unif_bound}
 \|\psi_{\e}\|_{L^\infty(0,T)} \le \|\kappa_{\e}\ast f \|_{L^{\infty}(0,T)}\le \|\kappa_{\e}\|_{L^1(0,T)} \|f\|_{L^{\infty}(0,T)}\le \frac{1}{\lm} \|f\|_{L^{\infty}(0,T)}.
% \psi_{\e}(t)
\end{equation}
%\begin{equation}\label{unif_bound}
%	\| \psi_{\e_n}\|_{L^\infty(0,T)}\le\frac{1}{\lambda} \|f\|_{L^\infty(0,T)},\quad n\in\mathbb{N}.
%\end{equation}
\bl{Now} define the extended maps $\bar{\psi}_n\colon\mathbb{R}\to\mathbb{R}_-$ by 
\[
\bar{\psi}_n(t)=\begin{cases}
	\psi_{\e_n}(t),&t\in [0,T],\\
	0,&\text{otherwise}.
\end{cases}
\]
To prove the lemma, by the Kolmogorov-Riesz-Fr\'{e}chet theorem (see e.g. Theorem 4.26 in \cite{Brez11}), it \bl{suffices to show} that 
\begin{equation}\label{translation}
	\lim_{h\to 0 }	\|\tau_h\bar{\psi}_n-\bar{\psi}_n\|_{L^1(\mathbb{R})}=0\quad \text{uniformly in }n,
\end{equation}
where $\tau_h$ \bl{is} the translation operator defined by $\tau_h g(x)=g(x+h)$ \bl{for an arbitrary function} $g\colon \mathbb{R}\to \mathbb{R}$. \vspace{2mm}\\
Consider the case $h>0$; when $h<T$,  by \eqref{RV_eq}, recalling \eqref{def_h} and \eqref{def_htilde},
\begin{align*}
	\notag	\tau_h\bar{\psi}_n(t)-\bar{\psi}_n(t)&=(\kappa_{\e_n}\ast f) (t+h)-(\kappa_{\e_n}\ast f )(t)
	\\\notag
	&\quad 
	+
	\int_{t}^{t+h}\kappa_{\e_n}(s)\Big(\frac{1}{2}\sigma^2 \bar{\psi}_n(t+h-s) +h(\bar{\psi}_n(t+h-s))\Big) \bar{\psi}_n(t+h-s)ds
	\\\notag
	&\quad +
	\int_{0}^{t}\kappa_{\e_n}(s)\phi_n(t-s;h) (	\tau_h\bar{\psi}_n(t-s)-\bar{\psi}_n(t-s))ds\\
	&:= \mathbf{I}_{n,h}(t) + \mathbf{II}_{n,h}(t)
	+
	(\kappa_{\e_n}\ast (\phi_n(\cdot;h)(\tau_h\bar{\psi}_n-\bar{\psi}_n))) (t),\quad t\in [0,T-h],
\end{align*}
where \bl{$\mathbf{I}_{n,h}(t)$ and $\mathbf{II}_{n,h}(t)$ refer to the first and second lines respectively on the right hand side here, and}
\begin{align}
\phi_n(t;h)=\frac{1}{2}\sigma^2(\tau_h\bar{\psi}_n(t)+\bar{\psi}_n(t))
+	\tilde{h}(\tau_h\bar{\psi}_n(t),\bar{\psi}_n(t)),\quad t\in \mathbb{R}.
\label{phi_n}
\end{align}
%It follows that the map 
Hence $\chi=\tau_h\bar{\psi}_n-\bar{\psi}_n -\mathbf{\upperRomannumeral{1}}_{n,h} - \mathbf{\upperRomannumeral{2}}_{n,h}$ solves the linear VIE
\[
\chi
=
\kappa_{\e_n}\ast (\phi_n(\cdot;h)(\tau_h\bar{\psi}_n-\bar{\psi}_n)) =\kappa_{\e_n}\ast (\phi_n(\cdot;h)\chi+\phi_n(\cdot;h)(\mathbf{\upperRomannumeral{1}}_{n,h} + \mathbf{\upperRomannumeral{2}}_{n,h}))
\]
on the interval $[0,T-h]$. 
%We observe that, by construction, 
 $\mathbf{\upperRomannumeral{1}}_{n,h},\, \mathbf{\upperRomannumeral{2}}_{n,h}$ and $\phi_n(\cdot;h)$ are continuous on $[0,T-h]$, so (\bl{given that $\phi_n(\cdot;h)$ is non-positive}), Theorem C.3 in \cite{AE19} implies that 
\begin{equation}\label{inter_1}
	|(\tau_h\bar{\psi}_n-\bar{\psi}_n)(t)|\le 
	|\mathbf{\upperRomannumeral{1}}_{n,h}(t) |+ |\mathbf{\upperRomannumeral{2}}_{n,h}(t)| + 
	(\kappa_{\e_n}\ast|\phi_n(\cdot;h)(\mathbf{\upperRomannumeral{1}}_{n,h} + \mathbf{\upperRomannumeral{2}}_{n,h})|)(t)
	,\quad t\in [0,T-h].
\end{equation}
%Observe that 
%We \bl{now} can apply Tonelli's theorem to compute
We compute
\begin{align}\notag\label{inter_11}
	&	\int_{0}^{T-h}\bigg(\int_{t}^{t+h}\kappa_{\e_n}(s)|f(t+h-s)|ds\bigg) dt=
	\int_{0}^{T-h}\bigg(\int_{0}^{T}1_{\{s<t+h\}}1_{\{s>t\}}\kappa_{\e_n}(s)|f(t+h-s)|ds\bigg) dt\\
	&\notag\qquad =\int_{0}^{T} \kappa_{\e_n}(s)\bigg(
	\int_{0}^{T-h}1_{\{t<s\}}1_{\{t>s-h\}}|f(t+h-s)|dt\bigg)ds \quad \bl{\text{(by Tonelli)}}
	\\
	&\qquad \le  \int_{0}^{h} \kappa_{\e_n}(s)\bigg(\int_{0}^{s}|f(t+h-s)|dt\bigg)ds 	
	+ \int_{h}^{T} \kappa_{\e_n}(s)\bigg(\int_{s-h}^{s}|f(t+h-s)|dt\bigg)ds \nn
	\\&\qquad \le  \frac{2}{\lambda} h \|f\|_{L^\infty(0,T)},
\end{align}
\bl{where $h$ appears in the final term since both inner integrals have range $\le h$}.
Thus, 
%by \cite[Theorem 2.2, Chapter 2]{Gr}, 
denoting by $\bar{f}=f 1_{[0,T]}\in L^\infty(\mathbb{R})$ and using Theorem 2.2, Chapter 2 in \cite{GLS90},
\begin{align*}
	&	\int_{0}^{T-h}|\mathbf{\upperRomannumeral{1}}_{n,h}(t)| dt \\&\qquad \le \int_{0}^{T-h} \int_{0}^{t}\bigg(\kappa_{\e_n}(s) |(\tau_hf-f)(t-s)|ds\bigg)dt
	+
	\int_{0}^{T-h}\bigg(\int_{t}^{t+h}\kappa_{\e_n}(s)|f(t+h-s)|ds\bigg) dt\\
	&\qquad 
	\le \bigg(\int_{0}^{T-h}\kappa_{\e_n}(s) ds \bigg) \|\tau_h\bar{f}-\bar{f}\|_{L^1(\mathbb{R})}
	+ \frac{2}{\lambda} h  \|f\|_{L^\infty(0,T)}
	\le \frac{1}{\lambda} \|\tau_h\bar{f}-\bar{f}\|_{L^1(\mathbb{R})} + \frac{2}{\lambda}h\|\bar{f}\|_{L^\infty(\mathbb{R})}.
\end{align*}
Since these estimates do not depend on $n\in \mathbb{N}$,  the continuity of the translation in $L^1(\mathbb{R})$ (see, for instance, Lemma 4.3 in \cite{Brez11}) yields that 
\begin{equation}\label{1step}
	\lim_{h\to 0+ }	\int_{0}^{T-h}|\mathbf{\upperRomannumeral{1}}_{n,h}(t)| dt =0	\quad \text{uniformly in }n.
\end{equation}
Given that $(\bar{\psi}_n)_n$ is bounded in $L^\infty(\mathbb{R})$ by \eqref{unif_bound} and $h(\cdot)$ (defined in \eqref{def_h}) is continuous on $\mathbb{R}$, the \bl{same} computations  as in \eqref{inter_11} (\bl{but} with $1$ in \bl{place} of $f$) show that 
\begin{equation}\label{inter_2}
	\lim_{h\to 0+ }	\int_{0}^{T-h}|\mathbf{\upperRomannumeral{2}}_{n,h}(t)| dt =0	\quad \text{uniformly in }n.
\end{equation}
Then (again by \eqref{unif_bound} and the continuity of $\tilde{h}$), there exists a constant $C>0$ such that 
\[
|\phi_n(t;h)|\le C,\quad t,\,h\in \mathbb{R},\,n\in\mathbb{N},
\]
where $\phi_n(\cdot;h)$ is defined in \eqref{phi_n}.
%It then follows from \cite[Theorem 2.2, Chapter 2]{Gr} that  
\bl{Therefore}
\begin{align*}
	\int_{0}^{T-h}	|(\kappa_{\e_n}\ast|\phi_n(\cdot;h)(\mathbf{\upperRomannumeral{1}}_{n,h} + \mathbf{\upperRomannumeral{2}}_{n,h})|)(t)|dt \le C \frac{1}{\lambda}\int_0^{T-h}(|\mathbf{\upperRomannumeral{1}}_{n,h}(t)| +|\mathbf{\upperRomannumeral{2}}_{n,h}(t)|) dt,
\end{align*}
whence (by \eqref{1step} and \eqref{inter_2}), 
\begin{equation}\label{inter_3}
	\lim_{h\to 0+} 	\int_{0}^{T-h}	|(\kappa_{\e_n}\ast|\phi_n(\cdot;h)(\mathbf{\upperRomannumeral{1}}_{n,h} + \mathbf{\upperRomannumeral{2}}_{n,h})|)(t)|dt =0\quad \text{uniformly in }n.
\end{equation}
Combining \eqref{1step}, \eqref{inter_2} and \eqref{inter_3} in \eqref{inter_1} we deduce that 
\begin{equation}\label{part1}
	\lim_{h\to 0+}\|\tau_h\bar{\psi}_{n}-\bar{\psi}_n\|_{L^1(0,T-h)}=0 \quad \text{uniformly in }n.
\end{equation}
On the interval $[-h,0]$ the maps  $\bar{\psi}_n$ equal 0, hence, by \eqref{unif_bound},
\begin{align*}
	\|\tau_h\bar{\psi}_{n}-\bar{\psi}_n\|_{L^1(-h,0)}=\int_{-h}^{0} |\tau_h\bar{\psi}_n(t)|dt \le
	\Big(\sup_n\|\bar{\psi}_n\|_{L^\infty(\mathbb{R})}\Big)h  \underset{h\to 0+}{\longrightarrow} 0 \quad \text{uniformly in }n.
\end{align*}
In a similar way, on the interval $[T-h,T]$ 
\[
\|\tau_h\bar{\psi}_{n}-\bar{\psi}_n\|_{L^1(T-h,T)}=\int_{T-h}^{T} |\bar{\psi}_n(t)|dt\le 	\Big(\sup_n\|\bar{\psi}_n\|_{L^\infty(\mathbb{R})}\Big)h  \underset{h\to 0+}{\longrightarrow} 0 \quad \text{uniformly in }n.
\]
The three previous equations yield
\begin{equation*}
	\lim_{h\to 0+ }	\|\tau_h\bar{\psi}_n-\bar{\psi}_n\|_{L^1(\mathbb{R})}=0\quad \text{uniformly in }n.
\end{equation*}
When $h<0$, assuming without loss of generality that $|h|<T$ we can simply write
\begin{align*}
	\|\tau_h\bar{\psi}_n-\bar{\psi}_n\|_{L^1(\mathbb{R})}&= \|\bar{\psi}_n\|_{L^1(0,|h|)}+ \|\tau_h\bar{\psi}_n\|_{L^1(T,T+|h|)} + 
	\|\bar{\psi}_n-\tau_h\bar{\psi}_n\|_{L^1(|h|,T)}\\
	&
	\le 2\Big(\sup_n\|\bar{\psi}_n\|_{L^\infty(\mathbb{R})}\Big)|h| +\|\tau_{|h|}\bar{\psi}_n-\bar{\psi}_n\|_{L^1(0,T-|h|)}  
	\underset{h\to 0-}{\longrightarrow} 0 \quad \text{uniformly in }n,
\end{align*}
where we use \eqref{part1} for the last limit.
Therefore \eqref{translation} is verified and the proof is complete.
\end{proof}
\subsection{Characterization of the limit points \bl{of $\psi_{\e}$}}\label{sec_conv}
\begin{lem}\label{lemma_conv_cont}
For every $T>0$ and $g\in L^1(0,T)$,
\[
	\lim_{\e\to 0}\int_{0}^{T} \Big|(\kappa_{\e}\ast g )(t)-\frac{1}{\lambda}g(t)\Big|dt=0,
\]
\bl{i.e.} $\kappa_{\e}\ast g$ converges to $\frac{1}{\lambda}g$ in $L^1(0,T)$ as $\e\to 0$.
\end{lem}
\begin{proof}
	\bl{Let} $c>0$. By the  continuity of the translation in $L^1(\mathbb{R})$ (see e.g. Lemma 4.3 in \cite{Brez11}),  there exists an $\eta=\eta(c)\in(0,T)$ such that, defining $\bar{g}=g1_{[0,T]}\in L^1(\mathbb{R})$, \bl{$\int_{0}^{T}|\bar{g}(t-s)-\bar{g}(t)|dt <c$}, and hence 
	\[
		\int_{s}^{T}|g(t-s)-g(t)|dt <c,
	\]
\bl{\noindent for $s\in (0,\eta)$}.  Then from Tonelli's theorem and some straightforward manipulations,
%	 {\bf Comment: here I use one of the estimates \bl{which one, I don't see that being used here?} on the integral of $\kappa_\e$ in your file, we must add it!}
	\begin{align*}
		&\int_{0}^{T} \Big|(\kappa_{\e}\ast g )(t)-\frac{1}{\lambda}g(t)\Big|dt \le 
		\int_{0}^{T} \bigg(\int_{0}^{t}\kappa_{\e}(s) |g(t-s)-g(t)|ds+|g(t)|\bigg(\frac{1}{\lambda}-\int_{0}^{t}\kappa_{\e}(s)ds\bigg)\bigg)dt 
		\\
		&\qquad=\bigg\{\int_{0}^{\eta}+\int_{\eta}^{T}\bigg\} \kappa_{\e}(s)\bigg(\int_{s}^{T}|g(t-s)-g(t)|dt\bigg)ds+ \int_{0}^{T}|g(t)|	\bigg(\frac{1}{\lambda}-\int_{0}^{t}\kappa_{\e}(s)ds\bigg)dt
		\\
		&\qquad
		\le 
		\frac{1}{\lambda}c+2\|g\|_{L^1(0,T)}\int_{\eta}^{T} \kappa_{\e}(s)ds + \int_{0}^{T}|g(t)|	\bigg(\frac{1}{\lambda}-\int_{0}^{t}\kappa_{\e}(s)ds\bigg)dt,
	\end{align*}
 and hence (by the dominated convergence theorem and \eqref{asymptotic}),
\[
	\limsup_{\e\to0}
	\int_{0}^{T} \Big|(\kappa_{\e}\ast g )(t)-\frac{1}{\lambda}g(t)\Big|dt
	\le \frac{1}{\lambda}c.
\]
Since $c$ can be chosen arbitrarily small the proof is complete.
\end{proof}\vspace{2mm}

\bl{Now recall the $\e=0$ solution $\psi_0$ defined in \eqref{lim_eq}, see Lemma \ref{lem_aggiunto}.  Then we have the following:}
\begin{lem}\label{lem_deltabehaviour}
Consider $T>0$ and  a sequence $(\epsilon_n)_n\subset (0,\infty)$ which converges to $0$. Suppose that there exists a non-positive function $\bar{\psi}\in L^1(0,T)\cap L^\infty (0,T)$ such that $\psi_{\e_n}\to \bar{\psi} $ in $L^1(0,T)$. Then $\bar{\psi}=\psi_0$ a.e. in $(0,T)$.
\end{lem}
\begin{proof} 
\bl{Recall our original VIE in \eqref{RV_eq}:
$\psi_{\e}=\kappa_{\e}\ast f + \kappa_\e \ast \bar{G}(\psi_{\e})$}.  Multiplying by $\lambda$, taking the difference with \eqref{lim_eq}, and adding and subtracting $\lambda( \kappa_{\e_n}\ast \bar{G}(\bar{\psi}))(t)$, we see that
\begin{align}\label{intermediate_conv1}
&\notag\lambda (\psi_{\e_n}(t)-\psi_{0}(t))\\&\quad = 	\lambda(\kappa_{\e_n}\ast f) (t) -f(t) \,+\, \lambda (\kappa_{\e_n} \ast (\bar{G}(\psi_{\e_n})-\bar{G}(\bar{\psi})))(t)
	\,+\,
	\lambda 	(\kappa_{\e_n} \ast \bar{G}(\bar{\psi}))(t)-\bar{G}(\psi_{0}(t))
\end{align}
for every $t\in [0,T]$.
By Lemma \ref{lemma_conv_cont}, considering that $\bar{G}(\bar{\psi}(\cdot))$ belongs to  $L^1(0,T)$ because $\bar{\psi}\in L^1(0,T)\cap L^\infty (0,T)=\bl{L^\infty (0,T)}$,
\[
\lim_{n\to\infty}\|\lambda(\kappa_{\e_n}\ast f)  - f\|_{L^1(0,T)}=0	\qquad \text{and}\qquad \lim_{n\to\infty}  \|\lambda(\kappa_{\e_n} \ast \bar{G}(\bar{\psi}))-\bar{G}(\bar{\psi}(\cdot))\|_{L^1(0,T)}=0.
\]
%where the limits are in the $L^1(0,T)-$sense. 
\bl{and, using the map $\tilde{h}$ defined in \eqref{def_htilde} in the proof of Lemma \ref{lemA3},}
% by \cite[Theorem 2.2, Chapter 2]{Gr}, \eqref{def_htilde} and \eqref{unif_bound},
\begin{align*}
&	\lambda\|\kappa_{\e_n}\ast (\bar{G}(\psi_{\e_n})-\bar{G}(\bar{\psi}))
	\|_{L^1(0,T)}  \lee 	\bl{\lambda\|\kappa_{\e_n}\|_{L^1(0,T)} \|\bar{G}(\psi_{\e_n})-\bar{G}(\bar{\psi})\|_{L^1(0,T)}}
%	= \lambda\| \psi_{\e_n}-\bar{G}(\bar{\psi}))\|_{L^1(0,T)}
	\\
	&\qquad 	\le 
	 \bl{\sup_{n\in\mathbb{N}}\bigg(	\frac{1}{2}\sigma^2  \|\psi_{\e_n}+\bar{\psi}\|_{L^\infty(0,T)}+
	\|\tilde{h}(\psi_{\e_n},\bar{\psi})\|_{L^\infty(0,T)}\bigg) \|\psi_{\e_n}-\bar{\psi}\|_{L^1(0,T)}} \\
	&\qquad 	\le 
	\bigg(	\frac{1}{2}\sigma^2  \Big(\frac{1}{\lambda}\|f\|_{L^\infty(0,T)}+\|\bar{\psi}\|_{L^\infty(0,T)}\Big)+ \sup_{n\in\mathbb{N}}
	\|\tilde{h}(\psi_{\e_n},\bar{\psi})\|_{L^\infty(0,T)}\bigg) \|\psi_{\e_n}-\bar{\psi}\|_{L^1(0,T)}\underset{n\to\infty}{\longrightarrow}0.
\end{align*}
%\bl{explain why the sup here is finite }.
Note that $\tilde{h}$ is continuous, and hence bounded in compact sets, and since  $\psi_{\e_n}$ and $\bar{\psi}$ are (uniformly) bounded, they take value in a compact set (ball), a.e., so the supremum in the final line is finite.
% The estimate on the sup then follows
Thus,  from \eqref{intermediate_conv1} we deduce that 
\begin{align*}
	&\lambda(\bar{\psi}(t)-\psi_{0}(t))
	= \bl{\bar{G}(\bar{\psi}(t))-\bar{G}(\psi_{0}(t))}=
	\bigg(\frac{1}{2}\sigma^2(\bar{\psi}(t)+\psi_0(t))	+	\tilde{h}(\bar{\psi}(t),\psi_0(t))\bigg)(\bar{\psi}(t)-\psi_{0}(t))
	,\\&\qquad \text{for a.e. }t\in (0,T).	
\end{align*}
This implies that $\bar{\psi}=\psi_0$ a.e. in $(0,T)$. Indeed, if there exists a subset $N\subset (0,T)$ with positive \bl{Lebesgue} measure where $\bar{\psi}\neq \psi_0$, then dividing the previous equation by $\bar{\psi}-\psi_0$ gives
\[
\lambda= \frac{1}{2}\sigma^2(\bar{\psi}(t)+\psi_0(t))+	\tilde{h}(\bar{\psi}(t),\psi_0(t)) <0\quad \text{ a.e. in }N,
\]
which is \bl{a contradiction since} $\lambda>0$. The proof is now complete.
\end{proof}

\subsection{Conclusion}
From Lemma \ref{lem_RFK}, we know that every sequence $(\psi_{\e_n})_n$ of solutions to \eqref{RV_eq} (where  $(\e_n)_n\subset(0,\infty)$ converges to 0 as $n\to \infty$) admits a convergent subsequence $(\psi_{\e_{n_k}})_k$ in $L^1(0,T)$. Since $(\psi_\e)_{\e>0}$ is a bounded family of (continuous) non-positive functions in $L^\infty(0,T)$, see \eqref{unif_bound}, the limit point of this subsequence belongs to $L^1(0,T)\cap L^\infty(0,T)$ and is non-positive, as well. \\
By Lemma \ref{lem_deltabehaviour} in Subsection \ref{sec_conv}, there exists a unique possible non-positive $L^1(0,T)-$limit point for $(\psi_{\e_{n_k}})_k$ in  $L^1(0,T)\cap L^\infty(0,T)$: $\psi_{0}$, the unique non-positive solution of \eqref{lim_eq}. Therefore, by the subsequence convergence principle  we conclude that 
\[
\lim_{\e\to 0}\psi_\e=\psi_0	\quad \text{in }L^1(0,T).
\]

%\sk \sk
%Cruder lower bound? Touch but not go beyond lower barrier
%hence
 %the RHS of \eqref{eq:Vint} is strictly negative for $t>0$,
%so we also see that $\psi_{\e}(t)\le 0$. 
\section{Laplace transform of hitting time to an upper barrier for a spectrally negative L\'{e}vy process}
\label{section:LT}
\renewcommand{\theequation}{C-\arabic{equation}}
\setcounter{equation}{0}
Let $X$ be a spectrally negative {one-dimensional} L\'{e}vy process, i.e. $\nu_X(0,\infty)=0$, where $\nu_X$ is the L\'{e}vy measure associated with $X$, and assume $X_0=0$. Suppose that $\nu_X$ satisfies 
\begin{equation}\label{finite_ex}
	\int_{(-\infty,-1)}|x|\nu_X(dx)<\infty;
\end{equation} 
by general properties of L\'evy processes (see e.g. Theorem 25.3 in \cite{Sato99}), \eqref{finite_ex} ensures that $\mathbb{E}[|X_t|]<\infty$ for every $t>0$.\\
In the next proposition, we establish a formula for the Laplace transform of the first hitting time of 
$X$ to upper (non-negative) barriers.
\begin{prop}\label{Prop_LT}
	Consider a spectrally negative {one-dimensional} L\'{e}vy process $X$ with   L\'evy measure $\nu_X$ satisfying \eqref{finite_ex}, with $X$ not identically zero. Suppose that 
	\[
	\gamma:=	\mathbb{E}[X_1]\ge0.
	\]
	For every $b\ge0$, denote by  $\tau_b$ the first hitting time of $X$ to $b$, i.e. $\tau_b= \inf \{t\ge 0 : X_t>b\}$, and define the function $V\colon\mathbb{R}_+\to\mathbb{R}_+$ by 
	\[
	V(p):=\frac{1}{2}(\sigma_X)^2 p^2 + \gamma p + \int_{\mathbb{R}_-}(e^{px}-1-px)\nu_X(dx),\quad p\ge0,
	\]
	where $\sigma^2_X\ge0$ denotes the Gaussian component of $X$.  
	Then for all $q\ge 0$
	\begin{equation}\label{stat_thm_LT}
		\mathbb{E}[e^{-q\tau_b}] = e^{-bV^{-1}(q)},
	\end{equation}
	where $V^{-1}$ is the inverse of $V$.
\end{prop}
\begin{proof}
	By Theorem 25.17 in \cite{Sato99},
	{for all $p\ge0$} we have  
	\begin{equation}\label{eq_LTAp}
		\log \Ex[e^{pX_t}]= tV(p),\quad t\ge0.
	\end{equation}
	Thus, $V$ is the logarithmic moment generating function (or cgf) of $X_1$ on $\mathbb{R}_+$. It then follows from Lemma 2.2.5 in \cite{DZ98} that $V$ is convex. Moreover, $V$ is continuous and differentiable, with 
	\[
	V'(p)=\gamma + (\sigma_X)^2p + \int_{\mathbb{R}_-}x(e^{px}-1)\nu_X(dx),\quad p\ge0.
	\]
	Since $V'>0$ on $(0,\infty)$, $V$ is increasing on $\mathbb{R}_+$ and $\lim_{p\to\infty}V(p)=\infty$.

	From the stationary and independent increments property one can  verify that the process $M=(M_t)_{t\ge0}$ given by $M_t:=e^{p X_t-V(p)t}$ is an $\mc{F}_t^X-$martingale. Indeed, for $0\le s\le t$,
	\begin{equation}\label{precise}
		\mathbb{E}[M_t|\mathcal{F}^X_s]
		=
		\mathbb{E}\Big[e^{p(X_t-X_s)}|\mathcal{F}^X_s\Big]
		e^{pX_s-V(p)t}=
		\mathbb{E}\Big[e^{p X_{t-s}}\Big]e^{pX_s-V(p)t}
		=M_s
		,
	\end{equation}
	where we  use \eqref{eq_LTAp} for the third equality.\\
	Now choose $p>0$.  Then applying the Optional Stopping Theorem to the bounded stopping time $t\wedge \tau_b$ we have
	\bq
	1~ \mathbb{E}[M_{t \wedge \tau_b}(1_{\{\tau_b\le t\}}+1_{\{\tau_b> t\}})] 
	&=& \mathbb{E}[e^{p b-V(p)\tau_b}  \,1_{\{\tau_b \le t \}}]  \,+\, 
	\mathbb{E}[e^{p X_t-V(p)t} 1_{\{\tau_b>t \}} ]\nn \,.
	%&=& \mathbb{E}(e^{-\half c^2 \tau}\cosh(c a) \, 1_{\tau \le t })  \,+\, \mathbb{E}(e^{-\half c^2 t }\cosh(c W_t ) \,1_{\tau>t }) \nn
	\eq 
	Here for  the second equality we use that $X_{\tau_b}=b$ when $\tau_b<\infty$ ($\mathbb{P}-$a.s.), because $X$ can only have negative jumps. 
	Using the monotone convergence theorem and that  $\lim_{t\to \infty} 1_{\{\tau_b\le t\}}= 1_{\{\tau_b<\infty\}}$ for the left term, and the bounded convergence theorem for the right term (with the bound $e^{p b}$, since $V>0$ on $(0,\infty)$), we can take the limit as $t \to \infty$ and take the limit inside the expectation to obtain
	\begin{equation*}
		1 = \mathbb{E}[e^{p b-V(p)\tau_b}  \,1_{\{\tau_b<\infty\}}]. 
	\end{equation*}
	%Since
	$V$ is a bijection from $\mathbb{R}_+$  onto itself (since $\gm\ge 0$), so we can re-write this as
	\begin{equation}\label{Eq_infinite}
		\mathbb{E}[e^{-q\tau_b}  1_{\{\tau_b<\infty\}}]= e^{-b V^{-1}(q)},\quad q>0. 
	\end{equation}
	Letting $q\searrow 0$ and using the bounded convergence theorem again, we see that
	$$
	\mathbb{P}(\tau_b<\infty)=\mathbb{E}[1_{\{\tau_b<\infty\}}]= e^{-b V^{-1}(0+)}, 
	$$
	where $V^{-1}(0+)=\lim_{q\searrow 0}V^{-1}(q)$. Since $V^{-1}$ is continuous on $\mathbb{R}_+$, we deduce that $V^{-1}(0+)=V^{-1}(0)=0$. Consequently, $\tau_b<\infty $ $\mathbb{P}-$a.s. and  \eqref{Eq_infinite} becomes \eqref{stat_thm_LT}, completing the proof.
\end{proof}
\begin{rem}
	When $\gamma >0$, for every $b\ge 0 $ the finiteness of the stopping time $\tau_b$ can be directly inferred from the LLN in Theorem 36.5 of \cite{Sato99}.
\end{rem}

\section{Brief formal derivation of the main idea in \cite{AAR25}}
\label{section:AAR25}
\renewcommand{\theequation}{D-\arabic{equation}}
\setcounter{equation}{0}

%Abi-Jaber form:
%\bq
%X_t &=& G_0(t)\,+\,\int_0^t K(t-s)Z_s ds\nn\,
%\eq
%with $K\in L^1$.
%\bq
%5e^{i u A^{\e}_t} &=&  e^{i u V_0 t\,+\,
%i u \int_0^t \kappa_{\e}(t-s) W_{A^{\e}_s}ds}\nn\,
%\eq

Consider a family of hyper-rough Heston models (with zero mean-reversion for simplicity) for which the quadratic variation of the log stock price satisfies
\bq
\la \log S^n \ra_t ~ X^n_t &=&  V_0 t\,+\,
%V_0 t\,+\,\sigma (I^{\al} B_{A_{.}})(t)~ V_0 t\,+\,\sigma
\Big(H_n+\half\Big)\sigma \int_0^t (t-s)^{H_n-\half} W_{X^n_s}ds\nn\,
\eq
%where $H_n\searrow -\half$ as $n\to \infty$.
%\bq
%K_n(t) ~ (H^n+\half)t^{H^n-\half}\nn
%\eq
for $H_n\in (-\half,1)$. We refer the reader to  \cite{JR20}, Section 7 in \cite{A21} and Section 5 in \cite{FGS21} for more on this model. From Lemma 2.4 in \cite{AAR25}, %\footnote{this lemma is particularly easy to check when $f$ is a polynomial.} 
we formally expect that
\bq
\lim_{H_n\searrow -\half }\Big(H_n+\half\Big) \sigma \int_0^t (t-s)^{H_n-\half} W_{X^n_s}ds~ \sigma W_{X_t}\nn\,,
\eq
%(rigourize with VIEs) 
where $X$ is the weak limit of $X^n$, so we expect $X$ to satisfy
% with no mean reversion and $\sigma=1$, we get
\bq
X_t&=& V_0 t\,+\, \sigma W_{X_t}\label{eq:HT}\,.
\eq
Now let 
\bq
Y_t&=&-t +\sigma W_t \label{eq:Y_t}
\eq and set $\tilde{X}_t=H_{-V_0 t}$, where $H_b=\inf \lb t:Y_t=b\rb$.  Then setting $t\mapsto \tilde{X}_t$ in \eqref{eq:Y_t}, we see that
\bq
-V_0 t&=&  -\tilde{X}_t\,+\,   \sigma W_{\tilde{X}_t} \label{eq:Kuw}\,
\eq
i.e. $\tilde{X}$ satisfies the same equation as $X_t$ in \eqref{eq:HT}.  Hence (using the notation/setup in Lemma 2.3 in \cite{AAR25}, i.e. $c=-V_0$, $b=\sigma$ and $a=-1$), we  deduce that $X$ is an Inverse Gaussian L\'{e}vy process with
parameters $(V_0,\frac{V_0^2}{\sigma^2})$.

%\nind \textbf{add mean reversion $\lm$}
\sk
\sk
To analyze this process with VIEs, using that $\frac{1}{\Gamma(\al)}=\al+O(\al^2)$ as $\al \to 0$ (i.e. as $H\to -\half)$, we see that
the usual rough Heston VIE (with $\rho=0$) takes the form 
%{\color{cyan}Ale: I removed $\al \Gamma(\al) $ from inside the integrals below (it multiplied $\phi^2$). Put it back if I am mistaken!}
%\bq
%\phi(t) &=& (1+O(\al)) \al\int_0^t (t-s)^{\al-1}(\half (-u^2+iu)\,+\, \al \sigma  \phi(s)^2)ds\eq
%But this tends to
\bq
\phi(t) &=& \frac{1}{\Gamma(\al)} \int_0^t (t-s)^{\al-1}\Big(-\half (u^2+iu)\,+\, \half\sigma^2   \phi(s)^2\Big)ds\nn \\
&=&(1+O(\al))\al \int_0^t (t-s)^{\al-1}\Big(-\half (u^2+iu)\,+\, \half \sigma^2   \phi(s)^2\Big)ds
\,\,\,\,\to\,\,\,\, -\half (u^2+iu) \,+\, \half \sigma^2 \phi(t)^2\nn\,
\eq
as $\al \to 0$ (again using Lemma 2.4 in \cite{AAR25}), which is just an algebraic equation for $\phi$.
If we ignore the linear term in $u$ for simplicity (i.e. ignore the drift of the log stock price), then the (relevant) solution to this equation is $\phi(t)=\frac{1}{\sigma^2}(1 - \sqrt{1 + \sigma^2 u^2})$, i.e. the smaller root as in the proof of Theorem \ref{cont_prop}.
\end{appendix}

%\begin{funding}

%\end{funding}


\begin{thebibliography}{99}

\bibitem{A21}
Abi Jaber, E., ``Weak existence and uniqueness for affine stochastic Volterra equations with $L^1$-kernels'',  \emph{Bernoulli, 27}(3), 1583-1615, 2021.



\bibitem{AA25} Abi Jaber, E. and E. Attal, ``Simulating integrated Volterra square-root processes and Volterra Heston models via Inverse Gaussian'', Preprint, 2025.

\bibitem{AAR25} Abi Jaber, E., E. Attal and M. Rosenbaum, ``From Hyper Roughness to Jumps as $H\to-\half$'', Preprint, 2025.

\bibitem{ACLP}
Abi Jaber, E., C. Cuchiero, M. Larsson and S. Pulido, ``A weak solution theory for stochastic Volterra equations of convolution type'', \emph{Annals of Applied Probability, 31}(6), 2924-2952, 2021.

\bibitem{ACPPS}
Abi Jaber, E., C. Cuchiero, L. Pelizzari, S. Pulido and S. Svaluto-Ferro, 
``Polynomial volterra processes'', \emph{Electronic Journal of Probability, 29}, 1-37, 2024.

\bibitem{AC24} Abi Jaber, E. and N. De Carvalho, ``Reconciling rough volatility with jumps'', 
\textit{SIAM Journal on Financial Mathematics, 15}(3), 785-823, 2024.

\bibitem{AE19} Abi Jaber, E. and O. El Euch, ``Multifactor approximation of rough volatility models'', \textit{SIAM
	Journal on Financial Mathematics, 10}(2), 309-349, 2019.


\bibitem{ALP19} Abi Jaber, E., M. Larsson and S. Pulido, “Affine Volterra processes”, \textit{Annals of Applied Probability, 29}(5),  3155-3200, 2019.


\bibitem{Alfonsi}
Alfonsi, A., ``Nonnegativity preserving convolution kernels. Application to Stochastic Volterra Equations in closed convex domains and their approximation'', \emph{Stochastic Processes and their Applications, 181},  104535, 2025.

\bibitem{AS}
Alfonsi, A. and G. Szulda, ``On non-negative solutions of stochastic Volterra equations with jumps and non-Lipschitz coefficients'',  \emph{Bernoulli, 31}(4), 2890-2915, 2025.
%\bibitem{AP07} Andersen, L.B.G. and V.V. Piterbarg, ``Moment Explosions in Stochastic Volatility Models'', \emph{Finance and Stochastics, 11}(1), 29-50, 2007.

\bibitem{A}
Applebaum, D., ``\emph{Lévy processes and stochastic calculus}'', Cambridge University Press, 2009.
%\bibitem{BER25} Ballotta, L., E. Eberlein  and  G. Rayée, ``The term structure of implied
%correlations between S\&P and VIX markets'', \textit{Frontiers of Mathematical Finance, 5}, 73-93, 2025.

%\textit{SIAM Journal
%on Financial Mathematics}, 15(3):785–823, 2024.

\bibitem{BBV} Barndorff-Nielsen, O. E., F. E. Benth and A. E. Veraart, ``Modelling energy spot prices by volatility modulated Lévy-driven Volterra processes'',  \emph{Bernoulli, 19}(3), 803-845, 2013.


\bibitem{Bill86} Billingsley, P., ``\emph{Probability and measure}'', John Wiley\&Sons, 2nd  edition, 1986.

\bibitem{BF}
Bondi, A. and F. Flandoli, ``On the Kolmogorov equation associated with Volterra equations and Fractional Brownian Motion'', \emph{Journal of Functional Analysis}, 111234,  2025.

\bibitem{BLP24} 
Bondi, A., G. Livieri and S. Pulido,
``Affine Volterra processes with jumps'', \textit{Stochastic Processes and their Applications, 168},
1042-64,  2024.

\bibitem{BP_feller}
Bondi, A. and S. Pulido, ``Feller's test for explosions of stochastic Volterra equations'', Preprint, 2024.


\bibitem{BPS24} 
Bondi, A., S. Pulido and S. Scotti, 
``The rough Hawkes Heston stochastic volatility model'', 
\textit{Mathematical Finance, 34}(4), 1197-1241, 2024.

%\bibitem{BG25}  Bourgey, F. and J. Gatheral, ``The SSR under Quadratic Rough Heston'', Preprint, 2025.

\bibitem{BL24} Boyarchenko, S. and S. Levendorskii, ``Correct implied volatility shapes and reliable pricing in the rough Heston model'', Preprint, 2024.



\bibitem{Brez11}
Brézis, H., ``\emph{Functional analysis, Sobolev spaces and partial differential equations}'' (Vol. 2, No. 3), New York: Springer, 2011.

%\bibitem{BF18}
%Brouste, A. and M. Fukasawa, ``Local asymptotic normality property for fractional Gaussian noise under high-frequency observations'', \emph{Annals of Statistics, 46}(5), 2045-2061, 2018.

%\bibitem{CD24}
%Cont, R. and P. Das,
%``Rough Volatility: Fact or Artefact?'',
%\textit{Sankhya: The Indian Journal of Statistics, Series B, 86}(1), 191-223, 2024.

\bibitem{CT04} Cont, R. and P. Tankov, “\emph{Financial modelling with Jump Processes}”, Chapman\&Hall, 2004.

\bibitem{CIR}
 Cox, J.C., J.E. Ingersoll and S.A. Ross, ``A theory of term structure of interest
rates'', \emph{Econometrica, 53}, 385-408, 1985.

\bibitem{Cuch22} Cuchiero, C., ``Modelling rough covariance processes'', Talk, February 2022.

\bibitem{CT20}
Cuchiero, C. and J. Teichmann. 	``Generalized Feller processes and Markovian lifts of stochastic Volterra processes: the affine case", \emph{Journal of evolution equations, 20}(4), 1301-1348, 2020.

\bibitem{DZ98} Dembo, A. and O. Zeitouni, “\emph{Large deviations techniques and applications}”, Jones and Bartlet publishers, Boston, 1998.

\bibitem{DFS}
Duffie, D., D. Filipovic  and  W. Schachermayer, ``\emph{Affine processes and application in finance}'' (No. t0281), National Bureau of Economic Research, 2002.

\bibitem{EFR18} El Euch, O., M. Fukasawa and M. Rosenbaum, ``The microstructural foundations of leverage effect and rough volatility'', \textit{Finance\&Stochastics, 12}(6), 241-280, 2018.


\bibitem{ER18} El Euch, O. and M. Rosenbaum, ``Perfect hedging in rough Heston models'', \textit{Annals of Applied Probability, 28}(6), 3813-3856, 2018.

\bibitem{ER19} El Euch, O. and M. Rosenbaum, ``The characteristic function of Rough Heston models'', \textit{Mathematical
	Finance, 29}(1), 3-38, 2019.

\bibitem{FGS21} Forde, M., S. Gerhold and B. Smith, ``Small-time, large-time and $H\to 0$ asymptotics for the rough Heston model'', \textit{Mathematical Finance, 31}(1), 203-241, 2021.

%\bibitem{FGS22}  Forde, M., S. Gerhold and B. Smith, ``Small-time VIX smile and the stationary distribution for the Rough Heston model'', Working paper, 2022.

%\bibitem{FJ11} Forde, M. and A. Jacquier, ``The large-maturity smile for the Heston model'', \textit{Finance and
%	Stochastics, 15}(4), 755-780, 2011.

%\bibitem{FS21} Forde, M and B. Smith, ``Rough Heston with jumps - joint calibration to SPX/VIX level and skew as $T\to 0$, and issues with the quadratic rough Heston model'', Preprint, 2021.

%\bibitem{FSV21} Forde, M., B. Smith and L. Viitasaari, ``Rough volatility and CGMY jumps with a finite history and the Rough Heston model - small-time asymptotics in the $k\sqrt{t}$ regime'', \textit{Quantitative Finance, 21}(4), 541-563, 2021.










%\bibitem{FG24} Friz, P. and J. Gatheral, ``Computing the SSR'', Preprint, 2024.




\bibitem{GK19} Gatheral, J. and M. Keller-Ressel, ``Affine forward variance models'', 
\textit{Finance and Stochastics, 23}, 501-533, 2019.

%\bibitem{GGP19} Gerhold, S., C. Gerstenecker and A. Pinter, ``Moment Explosions In The Rough Heston Model'', \textit{Decisions in Economics and Finance, 42}(2), 575-608, 2019.

\bibitem{GLS90}
			Gripenberg, G., S.O. Londen and O. Staffans, ``\emph{Volterra integral and functional equations}'' (No. 34), Cambridge University Press, 1990.

\bibitem{HA}
Hamaguchi, Y., ``Markovian lifting and asymptotic log-Harnack inequality for stochastic Volterra integral equations'', \emph{Stochastic Processes and their Applications, 178}, 2024, 104482.
%\bibitem{HS21}
%Han, X. and A. Schied,
%``The roughness exponent and its model-free estimation'',  \emph{Annals of Applied Probability, 35}(2), 1049-1082, 2025.

\bibitem{JS}
Jacod, J. and A. Shiryaev, ``\emph{Limit theorems for stochastic processes}'', Vol. 288, Springer Science \& Business Media, 2003.

%\bibitem{Jai15}
%Jaisson, T., ``Market impact as anticipation of the order flow imbalance'',
%\textit{Quantitative Finance, 15}(7), 1123-1135, 2015.

\bibitem{JR16} Jaisson, T. and M. Rosenbaum, ``Rough fractional diffusions
as scaling limits of nearly unstable heavy tailed Hawkes processes'', \textit{Annals of Applied
Probability, 26}(5), 2860-2882, 2016.

\bibitem{JR20} Jusselin, P.  and M. Rosenbaum, ``No-arbitrage implies power-law market impact and rough volatility'',
\textit{Mathematical Finance, 30}(4), 1309-1336, 2020.


\bibitem{KKR13} Kuznetsov, A., A. Kyprianou and V. Rivero, ``The theory of scale functions for spectrally negative L\'{e}vy processes'', Review article, 2013.

%\bibitem{Lew00} Lewis, A.L., ``\emph{Option valuation under stochastic volatility}'', Finance Press, 2000.

\bibitem{MS}
Mytnik, L. and T. S. Salisbury, ``Uniqueness for Volterra-type stochastic integral equations'', Preprint, 2015.

%\bibitem{protter}	
%{Protter, P. E.,} ``\emph{Stochastic integration and differential equations}'', Springer, Berlin, Heidelberg, 2005.

%\bibitem{Rom22} R\o mer, S., 
%``Hybrid multifactor scheme for stochastic Volterra
%equations with completely monotone kernels'', Preprint, 2022.

\bibitem{Sato99} Sato, K., ``\emph{L\'{e}vy Processes and Infinitely Divisible Distributions}'', Cambridge University Press, 1999.

%\bibitem{Sin98} Sin, C., ``Complications with Stochastic Volatility Models'', \textit{Advances in Applied Probability, 30}, 256-268, 1998.

%\bibitem{Syz23}
%Szymanski, G.,  
%``Optimal Estimation of the Rough Hurst parameter in additive noise'',  
%\textit{Electronic Journal of Statistics}, 2023.

\bibitem{Wpaper}
Whitt, W., ``Some useful functions for functional limit theorems'', \emph{Mathematics of operations research, 5}(1),  67-85, 1980.

\bibitem{W}
Whitt, W.,      ``\emph{Stochastic-process limits: an introduction to stochastic-process limits and their application to queues}'', New York, NY: Springer New York, 2002.

\bibitem{Z}
Zhang, X., ``Stochastic Volterra equations in Banach spaces and stochastic partial differential equation'', \emph{Journal of Functional Analysis, 258}(4), 1361-1425, 2010.
\end{thebibliography}
\end{document}